\documentclass[11pt]{amsart}
\usepackage{amssymb}
\usepackage{amsthm}

\usepackage{epsfig}
\usepackage{color}
\usepackage{verbatim}

\usepackage{tikz}

\usepackage{amsmath}
\usepackage{hyperref}
\usepackage{xypic}
\usepackage{amscd}
\usepackage{color}

\newfont{\sheaf}{eusm10 scaled\magstep1}

\newtheorem{thm}{Theorem}[section]
\newtheorem{cor}[thm]{Corollary}
\newtheorem{lemma}[thm]{Lemma}
\newtheorem{prop}[thm]{Proposition}
\newtheorem{defn}[thm]{Definition}
\newtheorem{claim}[thm]{Claim}

\theoremstyle{definition}
\newtheorem{remark}[thm]{Remark}

 \newtheorem{notation}[thm]{Notation}
 \newtheorem{example}[thm]{Example}

\DeclareMathOperator{\Aut}{Aut}

\DeclareMathOperator{\Pic}{Pic}

\DeclareMathOperator{\Jac}{Jac}

\def\min{\operatorname{min}}

\def\max{\operatorname{max}}

\def\c1{\operatorname{c_1}}
\def\c2{\operatorname{c_2}}

\def\Sym{\operatorname{Sym}}

\def\s{\mathfrak{s}}
\def\f{\mathfrak{f}}

\def\e{\mathbf{e}}

\def\c{\mathfrak{P}}

\def\ZZ{{\mathbb Z}}

\def\QQ{{\mathbb Q}}
\def\PP{{\mathbb P}}

\def\DD{{\mathbb D}}
\def\GG{{\mathbb G}}

\def\G{{\mathcal G}}

\def\M{{\mathcal M}}

\def\O{{\mathcal O}}

\def\T{{\mathcal T}}
\def\H{{\mathcal H}}
\def\F{{\mathcal F}}
\def\K{{\mathcal K}}

\def\V{{\mathcal V}}

\def\C{{\mathcal C}} 
\def\K{{\mathcal K}}
\def\X{{\mathcal X}}

\def\x{\times}                   % product (fiber)
\def\cong{\simeq}

\def\+{\oplus}                   % direct sum
\def\*{\otimes}                  % tensor product
\def\Aut{\operatorname{Aut}}
\def\id{\operatorname{id}}

\def\Pic{\operatorname{Pic}}

\def\bP{{\mathbb P}}

\makeatletter
\@namedef{subjclassname@2020}{\textup{2020} Mathematics Subject Classification}
\makeatother

\begin{document}

\title[Brill-Noether loci of pencils with prescribed ramification]{Brill-Noether loci of pencils with prescribed ramification on moduli of curves and on Severi varieties on $K3$ surfaces}

\author[A.~L.~Knutsen]{Andreas Leopold Knutsen}
\address{A.~L.~Knutsen, Department of Mathematics, University of Bergen,
Postboks 7800,
5020 Bergen, Norway}
\email{andreas.knutsen@math.uib.no}

\author[S.~Torelli]{Sara Torelli}
\address{S.~Torelli,  Dipartimento di Matematica, Politecnico di Milano, Piazza Leonardo da Vinci 12, 20133 Milano, Italy}
\email{sara.torelli7@gmail.com}

\begin{abstract}
  Under the assumption that the {\it adjusted Brill-Noether number} $\widetilde{\rho}$ is at least $-g$, we prove that the Brill-Noether loci in $\M_{g,n}$ of pointed curves carrying pencils with prescribed ramification at the marked points have a component of the expected codimension with pointed curves having Brill-Noether varieties of pencils of the minimal dimension. As an application, the map from the Hurwitz scheme to $\M_g$ is dominant if $n+\widetilde{\rho} \geq 0$ and generically finite otherwise, settling a variation of a classical problem of Zariski.

  In the second part of the paper,
  we study the analogous loci of curves in Severi varieties on $K3$ surfaces, proving existence of curves with non-general behaviour from the point of view of Brill-Noether theory. This extends previous results of Ciliberto and the first named author to the ramified case. We apply these results to study correspondences and cycles on $K3$ surfaces in relation to Beauville-Voisin points and constant cycle curves.
\end{abstract}


\maketitle

\section{Introduction}

The theory of linear series on projective curves with prescribed ramification has been intensively studied since the appearance in the eighties of the groundbreaking results of Eisenbud and Harris \cite{EH} see, e.g., \cite{CPS,COP,CMPT,FP,Fa,FL,FT1,FT2,JPPZ17,Os2,Os,Tev,Te23}. There is an {\it adjusted Brill-Noether number} $\widetilde{\rho}$ (see \eqref{eq:ABN}) generalizing the classical Brill-Noether number $\rho(g,r,d)$, which in addition to $g,r,d$ also depends on the ramification profile. In \cite[Thm. 4.5 and its Rmk.]{EH} it is proved that a general pointed genus $g$ curve carries a $g^r_d$ with prescribed ramifications at the marked points only if $\widetilde{\rho}\geq 0$. In contrast to the unramified case, the condition becomes sufficient provided that an intersection number of Schubert classes is nonzero, and, if so,  the Brill-Noether loci have dimension $\widetilde{\rho}$. For pencils the conditions hold true by Osserman's results \cite[Thms. 2.4 and 2.6]{Os}. 

For negative $\widetilde{\rho}$ comparatively little is known about the loci of curves in the moduli space of $n$-pointed curves $\M_{g,n}$ carrying linear series with prescribed ramification, for instance their dimensions, and the dimension of the Brill-Noether varieties of the general members, even in the case of pencils. By contrast, in the classical case without ramification with $\rho<0$, it is known that the locus of curves in $\M_g$ carrying a $g^1_d$  is irreducible of codimension $-\rho$ and that its general element carries finitely many $g^1_d$s.

The first main result of this paper generalizes the classical results on Brill-Noether theory of {\it pencils} to the case of prescribed ramification, under the condition $\widetilde{\rho} \geq -g$. In Theorem \ref{thm:main_mod}, we prove that the moduli space of $n$-pointed curves with a $g^1_k$ with prescribed ramification has codimension $\min\{0,-\widetilde{\rho}\}$ in $\M_{g,n}$, and that its general member carries a
$\max\{0,\widetilde{\rho}\}$-dimensional family of such pencils.
When $\widetilde{\rho} \geq 0$, our result provides an alternative proof of the  aforementioned existence results of  Eisenbud-Harris \cite[Thm. 4.5 and its Rmk.]{EH} and Osserman \cite[Thms. 2.4 and 2.6]{Os}, not relying on Schubert calculus. Most importantly, our result establishes the case of negative adjusted Brill-Noether number $\widetilde{\rho}$. 

As a significant application, under the condition $\widetilde{\rho} \geq -g$, in Theorem \ref{thm:mainHur} we  prove that the natural forgetful map from the Hurwitz space of covers with prescribed ramifications to $\M_g$ is dominant precisely when $n+\widetilde{\rho}\geq 0$. This answers a variation of the classical conjecture of Zariski \cite{Zar} expecting the map from the Hurwitz space with prescribed monodromies to $\mathcal{M}_g$ to be dominant as soon as the dimension of the target is smaller and the covers are non-composed. Such a variation has already been intensively studied in the past years and partial results can be found in some of the above mentioned papers.

To prove the above results we use degenerations to nodal rational curves, performed in \S \ref{sec:proofs_K3free}. A refined version of this degeneration, performed in \S \ref{sec:rew}-\ref{sec:proofs}, allows us to exhibit curves having the desired properties as normalizations of nodal curves on $K3$ surfaces. In fact, in Theorem \ref{thm:main} we find loci of {\it nodal} curves on $K3$ surfaces whose normalizations carry a family of pencils with desired prescribed ramification of the expected dimension. We believe this result is of independent interest. For smooth curves, this extends to the ramified case Lazarsfeld's famous result stating that smooth curves on general K3 surfaces are Brill-Noether general. For nodal curves, this extends the results without ramification in \cite{CK} providing non-Brill-Noether general behavior.

The last results of the paper, proved in \S \ref{sec:Zk}, concern correspondences and cycles on K3 surfaces, and their relation to tautological points in $\M_{g,2}$. Indeed, as an application of Theorem \ref{thm:main}, we solve in full generality the problem of nonemptiness of   
components in the fiber over $0$ of the difference map $X \times X\to CH_0(X)$ studied by Mumford \cite{Mum}, obtained as $n$-torsion points in the Jacobian of curves in $X$, for a very general K3 surface $X$ (see Theorem \ref{thm:main_cycle}). We furthermore analyze these components with respect to Beauville-Voisin points and constant cycle curves (see Corollary \ref{cor:BV}). In view of Pandharipande and Schmitt's result \cite{PS19} relating Beauville-Voisin points and tautological points, we construct cycles in $\M_{g,2}$ with a dense set of tautological points (see Corollary \ref{cor:corBV}). 

In the next section we will present the main results of this paper more precisely, and explain all necessary notation.

\vspace{0.2cm} {\it Acknowledgments.} This collaboration started during the
conference {\it Algebraic Geometry in L'Aquila} in July 2023, funded by PRIN 2017 - 2017SSNZAW: "Moduli Theory and Birational Classification". We warmly thank the organizers for the wonderful conference. 
A.L.K. was partially supported 
	by the Trond Mohn Foundation's project ``Pure Mathematics in
	Norway''. S.T. was partially supported by PRIN 2022 "Moduli spaces and Birational Geometry" and by INdAM-GNSAGA. To conclude, we thank the referee for the careful reading of the paper and useful suggestions.

\section{Setting and main results of the paper}

In this section we will introduce notation  and present  the main results of the paper. We will divide the presentation in three parts:
\begin{itemize}
\item results concerning Brill-Noether theory of pencils with prescribed ramification on pointed curves (\S \ref{sec:intro_mod}),
\item results concerning Brill-Noether theory of pencils with prescribed ramification on normalizations of pointed nodal curves on $K3$ surfaces (\S \ref{sec:intro_K3}),
  \item applications to cycles on $K3$ surfaces and tautological points in the moduli space of $2$-pointed curves (\S \ref{sec:intro_tor}).
\end{itemize}

Throughout the paper, $g,k,n,e_1,\ldots,e_n$ will be integers satisfying
\begin{equation}
  \label{eq:setting1}
  g \geq 1, \; k \geq 2, \; 2 \leq e_i \leq k \; \;\mbox{for all} \;\; i \in \{1,\ldots,n\},
\end{equation}
and we define
\begin{equation}
  \label{eq:setting2}
  \e:=(e_1,\ldots,e_n) \;\; \mbox{and} \;\; e:=\sum e_i.
\end{equation}

 \subsection{Moduli spaces of pointed curves with prescribed ramification} \label{sec:intro_mod}

 Let $ \M_{g,n}$ denote the moduli space of smooth $n$-pointed genus $g$ curves. For  $(C,x_1,\ldots, x_n) \in \M_{g,n}$ and  $\e=(e_1,\ldots,e_n)$ we let
\[ G^1_k(C,(x_1,e_1),\ldots,(x_n,e_n))\]
be the variety parametrizing $g^1_k$s on $C$ with ramification order $e_i$ at $x_i$. 
Note that
\[
e \leq 2(k-1+g)+n
 \]
 (recall \eqref{eq:setting2}) by Riemann-Hurwitz.
 In the classical (unramified) setting, $G^1_k(C)$ denotes the variety parametrizing $g^1_k$s on $C$. For a general curve $C$, a necessary and sufficient condition for $G^1_k(C)\neq \emptyset$ is $\rho(g,1,k)\geq 0$, where $\rho$ is the Brill-Noether number. 

Let $\M_{g,k,\e} \subset \M_{g,n}$ be the locus of $n$-pointed curves $(C,x_1,\ldots,x_n)$ such that \linebreak $G^1_k(C,(x_1,e_1),\ldots,(x_n,e_n)) \neq \emptyset$.  A necessary and sufficient condition for such nonemptiness on a general $n$-pointed curve is that $\widetilde{\rho}=\widetilde{\rho}(g,1,k;e_1,\ldots,e_n)\geq 0$, where
\begin{equation} \label{eq:ABN}
  \widetilde{\rho}=\widetilde{\rho}(g,1,k;e_1,\ldots,e_n):=\rho(g,1,k)-\sum_{i=1}^n(e_i-1)=2k-2-g-e+n
\end{equation}
is the {\it adjusted Brill-Noether number}. In other words, $\M_{g,k,\e} =\M_{g,n}$ {\it if and only if} $\widetilde{\rho}=\widetilde{\rho}(g,1,k;e_1,\ldots,e_n)\geq 0$. In contrast to the classical case (with no ramification), the "if statement" relies on the computation of a nonzero intersection number of Schubert classes in the Chow ring of the Grassmannian $\GG(d -2, d)$ provided by Ossermann, cf., \cite[Thm. 5.42]{HM} and \cite[Thms. 2.4 and 2.6]{Os}. 

We denote by $\G^1_{g,k,\e}$ the variety parametrizing pairs 
$\left((C,x_1,\ldots,x_n),\mathfrak{g}\right)$ such that $(C,x_1,\ldots,x_n) \in \M_{g,k,\e}$ and $\mathfrak{g} \in G^1_k(C,(x_1,e_1),\ldots,(x_n,e_n))$ and by
\begin{equation}\label{eq:kappa}
\kappa_{g,k,\e}:\G^1_{g,k,\e} \longrightarrow \M_{g,k,\e} 
\end{equation}
the forgetful map, with fiber over $(C,x_1,\ldots,x_n)$ being $G^1_k(C,(x_1,e_1),\ldots,(x_n,e_n))$. This is related to the Hurwitz moduli space
\[ \H_{g,k,\e}:=\H_{g,k}(e_1,\ldots,e_n,2^r) \]
parametrizing degree $k$ covers
$X \to \PP^1$ with $X$ a smooth projective curve of genus $g$ with points of ramification $e_i$ over $n$ ordered points of the target, and simple ramification over $r:=2(k-1+g)+n-e$ ordered points of the target completing the ramification profile\footnote{Recall that two covers
$f:X \to \PP^1$ and $f':X' \to \PP^1$ are isomorphic if there exists a commutative diagram
\[ \xymatrix{ X \ar[r]^{\simeq} \ar[d]_f & X' \ar[d]^{f'} \\
              \PP^1 \ar[r]^{\simeq}  &  \PP^1 
  }
  \]
respecting all the markings.}, cf., e.g., \cite{HMu,ACV,Mo,FP}. We have a natural finite\footnote{Non-injectivity is due to the possible presence of automorphisms of pointed curves not descending to the base $\PP^1$ of the $g^1_k$.}  forgetful map dominating each component\footnote{The image of $\lambda_{g,k,\e}$ is the open dense sublocus of $\G^1_{g,k,\e}$ consisting of pairs $\Big((X,P_1,\ldots,P_n),\mathfrak{g}\Big)$ such that $P_1,\ldots,P_n$ lie in distinct members of $\mathfrak{g}$ and the ramification is otherwise simple and in
$2(k-1+g)+n-e$
other members of $\mathfrak{g}$. The density follows from a generalization of the theory of Harris and Mumford on admissible covers \cite{HMu} (see also \cite{FP,ACV,CMR}).} 
\begin{equation} \label{eq:lambda}
  \lambda_{g,k,\e}: \H_{g,k,\e} \longrightarrow \G^1_{g,k,\e}.
  \end{equation}
By 
Riemann's Existence Theorem\footnote{Indeed,  the dimension of each component of $\H_{g,k,\e}$ equals the number of marked points, plus the number of free branch points minus $\dim \Aut \PP^1$, which equals
\begin{eqnarray*}
  n+\left(2g-2+2k-\sum_{i=1}^n(e_i-1)\right)-3=n+(2g-2+2k-e+n)-3\\
  =2g-5+2k+2n-e
  =3g-3+n+\widetilde{\rho}.
\end{eqnarray*}} (cf., e.g., \cite[III, Cor. 4.10]{Mi}),
\begin{equation}
  \label{eq:dimG1k}
  \H_{g,k,\e} \;\; \mbox{and} \;\; \G^1_{g,k,\e} \;\; \mbox{are equidimensional of dimension} \;\; (3g-3+n)+\widetilde{\rho}, \mbox{if nonempty}. 
\end{equation}
This, along with the classical result \cite[Thm. 4.5]{EH} of Eisenbud and Harris when $\widetilde{\rho} \geq 0$, implies that all components of $G^1_k(C,(x_1,e_1),\ldots,(x_n,e_n))$ have dimension at least $ \max\{0,\widetilde{\rho}\}$, and if equality holds for some $(C,x_1,\ldots,x_n)$, then
$\M_{g,k,\e}$ has a component of
codimension
\[ \max\{0,-\widetilde{\rho}\}\]
in $\M_{g,n}$. 
We call this 
the {\it expected codimension} of $\M_{g,k,\e}$. 
One of the main results of this paper proves the existence of a component of the expected codimension under the additional condition
  $\widetilde{\rho} \geq -g$:

\begin{thm} \label{thm:main_mod}
  In the setting \eqref{eq:setting1}-\eqref{eq:setting2}, assume that $\widetilde{\rho}(g,1,k;\e) \geq -g$. Then 
  $\M_{g,k,\e}$ has a component of the expected codimension $\max\{0,-\widetilde{\rho}\}$ in $\M_{g,n}$
  whose general member $(C,x_1,\ldots,x_n)$ satisfies
\[ \dim G^1_k(C,(x_1,e_1),\ldots,(x_n,e_n))=\max\{0,\widetilde{\rho}\}.\]
\end{thm}

\begin{remark} \label{rem:main_mod}
  Since the general member in any component of $\G^1_{g,k,\e}$ is base point free,  has simple ramification elsewhere and all ramification points lie in distinct fibers, the same properties are true for the general member in any component of \linebreak $G^1_k(C,(x_1,e_1),\ldots,(x_n,e_n))$ as in the theorem.
\end{remark}

 As a consequence, we also solve the classical problem regarding the dominance or generic finiteness of the forgetful map from the Hurwitz scheme $\H_{g,k,\e}$ to $\M_g$:
\[ \xymatrix{
    \H_{g,k,\e} \ar[r]^{\lambda_{g,k,\e}} & \G^1_{g,k,\e} \ar[r]^{\kappa_{g,k,\e}} & \M_{g,k,\e} \ar[r]^{\phi_{g,k,\e}} & \M_{g}
  }\] 
(where $\lambda_{g,k,\e}$ and $\kappa_{g,k,\e}$ are as in \eqref{eq:lambda} and \eqref{eq:kappa}, respectively, and $\phi_{g,k,\e}$ is the forgetful map). 
Since $\dim \H_{g,k,\e}= \dim \G^1_{g,k,\e} =\dim \M_g+(n+\widetilde{\rho})$ one may expect, as best case scenario, this map to be dominant if $n+\widetilde{\rho} \geq 0$ and generically finite (on some component) if $n+\widetilde{\rho} < 0$, so that, in the latter case, the codimension of its image is $-(n+\widetilde{\rho})$. We prove that, under the assumption $\widetilde{\rho} \geq -g$, this best case scenario actually happens:

\begin{thm} \label{thm:mainHur}
  In the setting \eqref{eq:setting1}-\eqref{eq:setting2}, assume that $\widetilde{\rho}(g,1,k;\e) \geq -g$.  Then the forgeful map
  $\H_{g,k,\e} \to \M_g$ is dominant if $n+\widetilde{\rho} \geq 0$ and
  generically finite on at least one component if $n+\widetilde{\rho} < 0$.
  \end{thm}

Both Theorems \ref{thm:main_mod} and  \ref{thm:mainHur} will be proved in \S \ref{sec:proofs_K3free}.

\subsection{Gonality loci on $K3$ surfaces}\label{sec:intro_K3}

Let  $(S,H)$  be a primitively  polarized complex $K3$ surface of genus $p \geq 2$.  Let $V_{|H|,\delta}\subseteq \vert H\vert$ be the \emph{Severi variety} of curves with $0 \leq \delta \leq p$ nodes, which is known to be nonempty of dimension $g:=p-\delta$, if $(S,H)$ is {\it general in moduli}. It has recently been proved in \cite{BLC} that $V_{|H|,\delta}$ is irreducible for general $(S,H)$ whenever $g \geq 4$ and connected whenever $g\in \{1,2,3\}$, but we will not make use of this.

For $C \in V_{|H|,\delta}$, denote by $\nu:C^{\nu} \to C$ its normalization map. Then $C^{\nu}$ has genus $g$. We have a {\it moduli map}
\[ m_g: V_{|H|,\delta} \longrightarrow \M_g \]
mapping $[C]$ to $[C^{\nu}]$.
We let $\M^1_{g,k}$ denote the locus of curves carrying a $g^1_k$, which is known to be irreducible of codimension $\max\{0,\rho(g,1,k)\}$.
We set
\[ V^k_{|H|,\delta}:=m_g^{-1}(\M^1_{g,k}), \]
that is, $V^k_{|H|,\delta}\subseteq V_{|H|,\delta}$ is the subvariety of curves  whose normalizations carry a $g^1_k$. If $(S,H)$ is general, then every smooth curve in $|H|$ is Brill-Noether general by Lazarsfeld's famous result \cite[Lem. 1.1]{Laz}, whence 
$V^k_{|H|,0}  \neq \emptyset$ if and only if $\rho(p,1,k) \geq 0$, that is, $p \leq 2k-2$; in fact $V^k_{|H|,0} = V_{|H|,0}$ when $p \leq 2k-2$. 
In \cite[Thm. 0.1]{CK} a complete description is given in the case $\delta>0$: if $(S,H)$ is general,  then
$V^k_{|H|,\delta} \neq \emptyset$ if and only if
\begin{equation}
    \label{eq:boundintro}
\rho(p,\alpha,k\alpha+\delta) \geq 0, \text{where}\; \; \alpha:=\Big\lfloor \frac{p-\delta}{2(k-1)}\Big\rfloor,
\end{equation}
equivalently
\[
    \delta \geq \alpha \Big(p-\delta -(k-1)(\alpha+1)\Big).
  \]
In particular, this provides, quite surprisingly, the existence of curves in $|H|$ whose normalizations carry pencils with {\it negative} Brill-Noether number, in contrast to the smooth case.
More precisely, it is proved that $V^k_{|H|,\delta}$ is equidimensional of dimension $\min\{2(k-1),g\}$ under this condition \cite[Rem. 5.6]{KLM}, and has an irreducible component whose general element is an irreducible curve $C$ such that $G^1_k(C^{\nu})$ is again as ``nice'' as possible, that is, $\dim G^1_k(C^{\nu})=\max\{0,\rho(g,1,k)\}$, and when $\rho(g,1,k) \leq 0$ 
(resp. $\rho > 0$), any (resp. the general) $g^1_k$ on $C^{\nu}$ has simple ramification and all nodes of $C$ are non-neutral with respect to it\footnote{A node on a curve is {\it non-neutral} with respect to a $g^1_k$ on the normalization of the curve if the two preimages of the nodes do not lie in the same member of the $g^1_k$, equivalently, if the $g^1_k$ does not descend to the $1$-nodal curve obtained by identifying the two points}.

 We will study the loci of {\it pointed curves in $V^k_{|H|,\delta}$ with prescribed ramification}:

  \begin{defn} \label{def:locus}
    We define $V_{|H|,\delta,n}$ to be the $(g+n)$-dimensional scheme parametrizing pointed curves $(C,x_1,\ldots,x_n)$ such that $C \in V_{|H|,\delta}$ and $x_1,\ldots,x_n$ are distinct points different from the nodes of $C$.

    Recalling \eqref{eq:setting1}-\eqref{eq:setting2}, we define 
    \[
    V^k_{|H|,\delta,\e}=V^k_{|H|,\delta,e_1,\ldots,e_n}\subseteq V_{|H|,\delta,n}
    \]
    to be the sublocus of $V_{|H|,\delta,n}$ such that 
$G^1_k(C^{\nu},(x_1,e_1),\ldots,(x_n,e_n)) \neq \emptyset$. 
\end{defn}

 In the definition we identify $x_i \in C$ with $\nu^{-1}(x_i) \in C^{\nu}$. Note that we have a moduli map
  \begin{equation} \label{eq:modgn}
    m_{g,n}: V_{|H|,\delta,n} \longrightarrow \M_{g,n},
    \end{equation}
    mapping $(C,x_1,\ldots,x_n)$ to $(C^{\nu},x_1,\ldots,x_n)$, and we then see that
    \[ V^k_{|H|,\delta,\e}=m_{g,n}^{-1}(\M_{g,k,\e}).\]
By \eqref{eq:modgn} and what we said in the previous subsection, 
the {\it expected codimension of $V^k_{|H|,\delta,\e}$ in $V_{|H|,\delta,n}$} is
$\max\{0,-\widetilde{\rho}\}$.

A {\it necessary} condition for the nonemptiness of
$V^k_{|H|,\delta,\e}$ is the nonemptiness of $V^k_{|H|,\delta}$, which as recalled above requires \eqref{eq:boundintro} to hold. 
Under the assumption $\widetilde{\rho} \geq -g$, we will prove that the latter is also a {\it sufficient} condition. More precisely, we will prove the following version of Theorems \ref{thm:main_mod} and \ref{thm:mainHur} on $K3$ surfaces. Since the properties in Remark \ref{rem:main_mod} are not automatic in this case, we will make use of the following:

\begin{notation} \label{not:star}
  A $g^1_k$ on a marked curve satisfies $(\star)$ if it is base point free,   has simple ramification outside the marked points, and all ramification points lie in distinct fibers.
  \end{notation}

\begin{thm} \label{thm:main}
  Let $(S,H)$ be a general primitively polarized $K3$ surface of genus $p \geq 2$, let $0 \leq \delta < p$ and set $g=p-\delta$. In the setting \eqref{eq:setting1}-\eqref{eq:setting2}, assume that $\widetilde{\rho}(g,1,k;\e) \geq -g$ holds. 

  Then $V^k_{|H|,\delta,\e} \neq \emptyset$ if and only if \eqref{eq:boundintro} holds, and whenever nonempty, it
has an irreducible component  $V^{k,\star}_{|H|,\delta,\e}$ of the expected codimension $\max\{0,-\widetilde{\rho}\}$ in $V_{|H|,\delta,n}$.
More precisely, the following hold:
 \begin{itemize}
 \item [(i)]
If $\widetilde{\rho} \geq 0$, then
 $V^{k,\star}_{|H|,\delta,\e}$ is a component of  
 $V_{|H|,\delta,n}$, and for a general $(C,x_1,\ldots,x_n) \in V^{k,\star}_{|H|,\delta,\e}$  the variety $G^1_k(C^{\nu},(x_1,e_1),\ldots,(x_n,e_n))$ has dimension $\widetilde{\rho}$, and the general member in any component satisfies $(\star)$ and all nodes of $C$ are non-neutral with respect to it;
\item [(ii)]
If $\widetilde{\rho} < 0$, then for a general
$(C,x_1,\ldots,x_n) \in V^{k,\star}_{|H|,\delta,\e}$  the variety \linebreak $G^1_k(C^{\nu},(x_1,e_1),\ldots,(x_n,e_n))$ is finite, and each member
satisfies $(\star)$ and all nodes of $C$ are non-neutral with respect to it. Moreover,
     \begin{itemize}
     \item [(ii-a)] if  $n+\widetilde{\rho}\geq 0$, then the forgetful map
       $V^{k,\star}_{|H|,\delta,\e}  \to V_{|H|,\delta}$ is dominant onto a component, and
     \item[(ii-b)] if $n+\widetilde{\rho}<0$, then the forgetful map is generically finite.
 \end{itemize}
 \end{itemize}
\end{thm}

Theorem \ref{thm:main} will be proved in \S \ref{sec:proofs}. 

The special case of $V^k_{|H|,\delta,k,k}$, that is, of curves carrying pencils with {\it two points of total ramification}, will find a special application in the study of cycles on $K3$ surfaces, as we will explain in the next subsection. In this case, $n=2$ and $e_1=e_2=k$, and one checks that the condition $\widetilde{\rho} \geq -g$ is automatically satisfied; we will actually prove that $V^k_{|H|,\delta,k,k}$ is {\it equidimensional} of dimension $2$, cf. \eqref{eq:uguali} and Theorem \ref{thm:main_cycle} below.

\subsection{Cycles of torsion differences of points on curves on $K3$ surfaces, Beauville-Voisin points and tautological points in $\M_{g,2}$} \label{sec:intro_tor}

Let  $(S,H)$  be a primitively  polarized complex $K3$ surface of genus $p \geq 2$ and $0 \leq \delta \leq p$. As above, for any irreducible curve $C \in |H|$, we denote by $C^{\nu}$ its normalization and $\nu:C^{\nu}\to C$ its normalization map.

The next definition generalizes \cite[Def. 3.1 and 3.4]{Tor23} to the case of nodal curves:

\begin{defn}\label{def:Zn}
  Fix an integer $k \geq 1$. We define the loci
 \[
    Z'^{\circ}_{k,\delta}(S,H):= \left\{ (C,p,q)\in V_{|H|,\delta,2} \; |\; k\cdot[p-q]=0\in \Jac(C^{\nu})\right\} \subset V_{|H|,\delta,2}
\] 
and $Z^{\circ}_{k,\delta}(S,H)$ its image in $S \x S$ by the forgetful map.
We denote their closures by  $Z'_{k,\delta}(S,H)$ and $Z_{k,\delta}(S,H)$.
\end{defn}

With this notation, the cases studied in \cite[Def. 3.1 and 3.4]{Tor23} are 
$Z'_{k,0}(S,H)$ and $Z_{k,0}(S,H)$.

It is immediate to see that 
\begin{equation}
  \label{eq:uguali}
  Z'^{\circ}_{k,\delta}(S,H)=V^k_{|H|,\delta,k,k}.
\end{equation}
Since one may check that the condition $\widetilde{\rho} \geq -g$ 
is satisfied in this case, Theorem \ref{thm:main} yields that $Z'_{k,\delta}(S,H)$ has a component of  dimension $2$  whenever \eqref{eq:boundintro}  is satisfied, and is otherwise empty. We can improve this result to give {\it equidimensionality}
and show that the dimension does not drop when forgetting the curves:

\begin{thm}\label{thm:main_cycle} Let $(S,H)$ be a general primitively polarized K3 surface of genus  $p\geq 2$. Let $0 \leq \delta  < p$  and  $k \geq 2$.  Then $Z'_{k,\delta}(S,H)$ and $Z_{k,\delta}(S,H)$ are equidimensional of  dimension $2$  whenever \eqref{eq:boundintro}  is satisfied, and are otherwise empty. 
 \end{thm}

If  $\delta=0$,  this says that $Z'_{k,0}(S,H)$ and $Z_{k,0}(S,H)$ are nonempty  with a $2$-dimensional component  for $p\leq 2(k-1)$  and are otherwise empty, which proves Conjecture \cite[Conj. 1]{Tor23}.  On the other hand, 
for any $k\geq 2$ there is a $\delta_0(p,k)\geq 0$ computed by \eqref{eq:boundintro} such that for any $\delta_0(p,k)\leq \delta  < p$, the varieties  $Z'_{k,\delta}(S,H)$ and $Z_{k,\delta}(S,H)$ are nonempty  (and $2$-dimensional). In other words, one can decrease $k$ by adding nodes,   as depicted in the following example for low genus $p$:

 \begin{example} \label{ex:p345} If $p=3$, Theorem \ref{thm:main_cycle} says that  $Z'_{k,0}(S,H)$ is nonempty if and only if $k \geq 3$ and that  $Z'_{2,1}(S,H)$ and $Z'_{2,2}(S,H)$ are nonempty, since one computes $\delta_0(3,2)=1$ from \eqref{eq:boundintro}. 
 
 If $p=4$, Theorem \ref{thm:main_cycle} says that  $Z'_{k,0}(S,H)$
 is nonempty  if and only if $k \geq 3$ and that  $Z'_{2,1}(S,H)$, $Z'_{2,2}(S,H)$ and $Z'_{2,3}(S,H)$ are nonempty,  since one again computes $\delta_0(4,2)=1$ from \eqref{eq:boundintro}.   

If $p=5$,  Theorem \ref{thm:main_cycle} says that  $Z'_{k,0}(S,H)$ is nonempty   if and only if  $k\geq 4$. One computes $\delta_0(5,3)=1$ from \eqref{eq:boundintro}, so the theorem says that $Z'_{3,\delta}(S,H)$ is nonempty for all $\delta\in \{1,2,3,4\}$. Similarly, one computes $\delta_0(5,2)=2$, so the theorem says that $Z'_{2,\delta}(S,H)$ is nonempty for all $\delta\in \{2,3,4\}$,  while $Z'_{2,1}(S,H)=\emptyset$. 
\end{example}

We next relate the loci $Z_{k,\delta}(S,H)\subset S \x S$ provided by Theorem \ref{thm:main_cycle} to the theory of Beauville-Voisin points and constant cycle curves. Recall that, by a classical result of Mumford \cite{Mum}, the Chow group $CH_0(S)$ of $0$-cycles on $S$  is huge. Nonetheless, it contains a distinguished class $c_S$ of degree $1$, called the Beauville-Voisin class, which is defined as the class of a point on a rational curve (cf. \cite[Thm. 1]{BV04}). 
A Beauville-Voisin point (BV-point in short) $p\in S$ is a point with class $c_S$. By \cite[Thm. 1.2]{M04} BV-points are dense in $S$ and the set is expected to be a union of curves. A curve whose points all define the class $c_S$ is called constant cycle curve (cf. \cite[Def. 3.1]{Huy14} and \cite{Voi15}). By \cite[Thm. 11.1]{Huy14} and \cite[Introd.]{CG22}, constant cycle curves are dense. 

In \cite{Tor23} $2$-cycles $Z$ such that $Z_*$ preserves BV-points and constant cycle curves are introduced, a notion we rephrase as:

\begin{defn} \label{def:pres}
  Let $S$ be a $K3$ surface. A 2-cycle $Z\subset S\times S$ is said to {\em preserve BV-points} (respectively, {\em constant cycle curves}) if
  \begin{itemize}
  \item[(i)] there is a lifting $Z_*:Z_i(S)\to Z_i(S)$ to the group $Z_i(S)$ of $i$-cycles of the natural morphism $Z_*:CH_i(S)\to CH_i(S)$, defined on irreducible subvarieties $W\subset S$ as $Z_*W={p_2}_*p_1^*W$, where $p_1,p_2:Z\to S$ are the two projections;
    \item[(ii)] for any BV-point $p \in S$ (resp., constant cycle curve $Y \subset S$), each point (resp. integral curve) in the support of the $0$-cycle $Z_*p$ (resp., in the $1$-cycle $Z_*Y$) is a  BV-point (resp. a constant cycle curve).
  \end{itemize}
\end{defn}

In \cite{Tor23} it is shown that $Z_{k,0}(S,H)_*$ preserves constant cycle curves. We can now extend this study to $\delta>0$. Indeed, as a consequence of Theorem \ref{thm:main_cycle}, we obtain:

\begin{cor}\label{cor:BV} In the setting of Theorem \ref{thm:main_cycle}, assume that \eqref{eq:boundintro} is satisfied. Then $Z_{k,\delta}(S,H)_*$ preserves BV-points and constant cycle curves (cf. Definition \ref{def:pres}).
  
Moreover, the points $(p,q)\in Z_{k,\delta}(S,H)$ defining the class $(c_S,c_S)$ are dense in every component.
\end{cor}

The latter property in the corollary allows us to give an application to tautological points in $\M_{g,2}$, in view of \cite[Thm. 1.5]{PS19} relating them to BV-points on curves on $K3$ surfaces.
Recall that, by \cite[Thm. 1.1]{GV01}, the degree-$0$ tautological group $R_0(\overline{\mathcal{M}_{g,n}}) \subseteq CH_0(\overline{\mathcal{M}_{g,n}})$ is always isomorphic to $\mathbb{Q}$, even though $CH_0(\overline{\mathcal{M}_{g,n}})$ is expected to be huge except for finitely many $(g,n)$ (cf. \cite[Spec. 1.1]{PS19}).  This setting is similar to the one on K3 surfaces. In \cite[Thm. 1.5]{PS19} it is proved that $(C,x_1,..,x_n)\in\overline{\M_{g,n}}$ is tautological if $C$ is a  smooth  curve of genus $g$ sitting in a K3 surface $S$, $x_i\in C$ is a BV-point for $1\leq i\leq n$ and $n\leq g$. 
As an application of this relation and Corollary \ref{cor:BV} we obtain:

\begin{cor} \label{cor:corBV}
  In the setting of Theorem  \ref{thm:main_cycle}, 
  assume that  $\rho(g,1,k) \geq 0$ . Then the set of tautological points in $m_{g,2}(Z'^{\circ}_{k,0}(S,H))=m_{g,2}(V^k_{|H|,0,k,k})\subset \M_{g,2}$  
(cf. \eqref{eq:modgn}) is dense in every component, for  $g\geq 2$. 
    \end{cor}

%A further interesting feature of this corollary is that  its relative version to the whole $19$-dimensional family of polarized $K3$ surfaces gives higher-dimensional cycles fully contained in $\M_{g,2}$ of this kind, see Corollary \ref{cor:relativeZ_k}.

 Theorem \ref{thm:main_cycle} and Corollaries \ref{cor:BV} and \ref{cor:corBV} will be proved in \S \ref{sec:Zk}.

\section{Proofs of Theorems \ref{thm:main_mod} and \ref{thm:mainHur}} \label{sec:proofs_K3free}

 Let $\overline{\M_{g,n}}$ be the Deligne-Mumford compactification of $\M_{g,n}$. 
We denote by $\overline{\M_{g,k,\e}} \subset \overline{\M_{g,n}}$ the compactification of $\M_{g,k,\e}$ in $\overline{\M_{g,n}}$. 

Let $\overline{\H_{g,k,\e}}$ be the (Deligne-Mumford) compactification  of the Hurwitz space $\H_{g,k,\e}$, which is given by a generalization of the classical theory of admissible covers, see, e.g., \cite{HMu,FP,ACV,CMR}. 
We denote by
\begin{equation}
  \label{eq:hj3} \overline{\mu}_{g,k,\e}: \overline{\H_{g,k,\e}} \longrightarrow \overline{\M_{g,k,\e}}
\end{equation}
the  forgetful map sending the domain of a cover to its stable reduction, which is an

extension of $\mu_{g,k,\e}:=\kappa_{g,k,\e}\circ \lambda_{g,k,\e}$ (cf. \eqref{eq:kappa} and \eqref{eq:lambda}). 
We recall that by the theory of admissible covers, $(X,P_1,\ldots,P_n)$ with $X$ irreducible lies in $\overline{\M_{g,k,\e}}$ if and only if the normalization $\nu:X^{\nu} \to X$ has the property that $X^{\nu}$ carries a $g^1_k$ in $G^1_k(X^{\nu},(P_1,e_1),\ldots,(P_n,e_n))$ such that each pair of points on
$X^{\nu}$ lying above a node of $X$ lies in the same fiber of the $g^1_k$. In other words, the morphism $X^{\nu} \to \PP^1$ defined by the $g^1_k$ factors through the normalization $X^{\nu}\to X \to \PP^1$. We call such a $g^1_k$ on $X^{\nu}$ a {\it descending} $g^1_k$ and denote by $G^1_k(X, (P_1,e_1),\ldots, (P_n,e_n))$ the locus of such descending $g^1_k$s. Then the theory of admissible covers yields 
 \begin{equation}
   \label{eq:hj2}
   \dim \overline{\mu}_{g,k,\e}^{-1}\left((X,P_1,\ldots,P_n)\right) = \dim G^1_k(X,(P_1,e_1),\ldots,(P_n,e_n)) \;\; \mbox{(on each component)}.
 \end{equation}

To prove Theorem \ref{thm:main_mod} we will find an $n$-pointed $g$-nodal rational curve $(X,P_1,\ldots,P_n)$ in $\overline{\M_{g,k,\e}}$ such that
$G^1_k(X,(P_1,e_1),\ldots,(P_n,e_n))$ has dimension $\max\{0,\widetilde{\rho}\}$. 
 
 To construct $g^1_k$s on $\PP^1$s descending to nodal models, we will work on $\Sym^2(\PP^1)$.  As customary, we identify $\Sym^2(\PP^1)$ with  $\PP^2$, in such a way that the diagonal $\Delta$ is a conic and each coordinate curve $\{x+y\;|\; y\in \PP^ 1\}$ is the tangent line $\ell_x$ to $\Delta$ at $2x$.

 Consider $\QQ=\PP^ 1\times \PP^ 1$ with the two projections $\pi_i: \QQ\to \PP^ 1$, $i=1,2$ and the line bundle $\O_{\QQ}(k,k):=\pi_1^ *(\O_{\PP^ 1}(k))\otimes \pi_2^ *(\O_{\PP^ 1}(k))$, for any $k \in \ZZ^+$.
Then $H^0(\O_{\QQ}(k,k))=H^ 0(\PP^1, \O_{\PP^ 1}(k))^ {\otimes 2}$. The two subspaces $\Sym^ 2(H^ 0(\PP^ 1, \O_{\PP^ 1}(k)))$ and 
$\wedge^ 2 H^ 0(\PP^ 1, \O_{\PP^ 1}(k))$ are invariant  (resp. anti-invariant) under the natural involution that exchanges the coordinates. Hence they are pull-backs of sections of line bundles on $\Sym^ 2(\PP^ 1)$, say $\O_k^+$ and $\O_k^-$, respectively.  For instance, one has 
\begin{equation} \label{eq:isom-}
  H^ 0(\Sym^ 2(\PP^ 1), \O_k^ -)\cong \wedge^ 2 H^ 0(\PP^ 1, \O_{\PP^ 1}(k))\cong
  H^ 0(\PP^2, \O_{\PP^ 2}(k-1)).
  \end{equation}
Let $\mathfrak{g}$ be any $g^1_k$ on $\PP^1$. Then it can be identified with 
a point   of the grassmannian $\mathbb G(1,k)\subset \PP( \wedge^ 2 H^ 0(\PP^ 1, \O_{\PP^ 1}(k)))$, which by \eqref{eq:isom-} can again be identified with the degree $k-1$ curve in $\PP^ 2$
\[ C_{\mathfrak{g}}= \{ W \in \Sym^2(\PP^1) \; | \; \mathfrak{g}(-W) \geq 0\}.\]

\begin{defn} \label{def:Fk}
  We denote the family of curves $\{C_{\mathfrak{g}}\}_{\mathfrak{g} \in G^1_k(\PP^1)}$ by $\F_k$.
\end{defn}

Note that $\F_k$  is irreducible of dimension $2(k-1)=\dim(\mathbb{G}(1,k))$.

\begin{lemma} \label{lemma:Cgred}
  The curve $C_{\mathfrak{g}}$ is reduced if and only if $\mathfrak{g}$ does not have multiple base points.
\end{lemma}

\begin{proof}
  If $P \in \PP^1$ is a base point of multiplicity $n \geq 2$ of $\mathfrak{g}$, then $C_{\mathfrak{g}}$ contains the line $\ell_P:=\{P+Q \; | \; Q \in \PP^1\} \subset \Sym^2(\PP^1)$ as a component, and the residual curve is $C_{\mathfrak{g}(-nP)}$. As the latter has degree $k-1-n$, the line $\ell_P$ has multiplicity $n$, whence
    $C_{\mathfrak{g}}$ is not reduced.

    If $\mathfrak{g}$ does not have multiple base points, then its general member consists of $k$ distinct points. Then $C_{\mathfrak{g}}$ intersects the line $\ell_P$ in $k-1$ distinct points. As $C_{\mathfrak{g}}$ has degree $k-1$, this shows that it is reduced.
\end{proof}

We will need the following: 

\begin{lemma} \label{lemma:P1suP1}
  If $e_1,\ldots,e_n,k$ are integers such that $2 \leq e_i \leq k$, $k \geq 2$ and $\widetilde{\rho}(0,1,k;\e) \geq 0$, then $G^1_k(\PP^1,(P_1,e_1),\ldots,(P_n,e_n)) \neq \emptyset$ for any distinct points $P_1,\ldots,P_n \in \PP^1$.  Moreover, every component has dimension at least $\widetilde{\rho}(0,1,k;\e)$
    and equality holds if $P_1,\ldots,P_n$ are general.  
\end{lemma}
\begin{proof}
  By the classical result \cite[Thm. 4.5]{EH} of Eisenbud and Harris, it suffices to prove nonemptiness. Moreover, it suffices to prove it in the case $\widetilde{\rho}(0,1,k;\e) =0$, that is, 
  $e=\sum e_i=2(k-1)+n$, by adding further simple ramification points if necessary to complete the ramification profile. The latter has been proved by Osserman \cite[Thms. 2.4]{Os}. We provide a proof independent of Schubert calculus.

  By Riemann's existence theorem (cf., e.g., \cite[III, Cor. 4.10]{Mi}) it suffices to prove that there are cycles $\sigma_1,\ldots,\sigma_n \in \Sym(k)$ of orders $e_1,\ldots,e_n$ generating a transitive subgroup and such that $\sigma_1\cdots\sigma_n=\id_k$ (the identity in $\Sym(k)$). We will do this by proving the following claim by induction on $k$:

  \begin{claim}
    There exist cycles $\sigma_1,\ldots,\sigma_n \in \Sym(k)$ of orders $e_1,\ldots,e_n$ such that
    \begin{itemize}
    \item[(i)] $\sigma_1\cdots\sigma_n=\id_k$,
    \item[(ii)] the subgroup generated by $\sigma_1,\ldots,\sigma_n$ is transitive,
      \item[(iii)] for all $i \in \{1,\ldots,n-1\}$, $\sigma_i$ and $\sigma_{i+1}$ are not disjoint. 
    \end{itemize}
  \end{claim}

  \begin{proof}[Proof of claim]
    We first treat the case $e_1=\cdots=e_n=2$, in which case $n=2k-2$. Then
    \[ \sigma_j=(j\;\;j+1), \;\; \sigma_{k-1+j}=\sigma_{k-j} \;\; \mbox{for} \;\;  j\in \{1,\ldots,k-1\}\]
    will do.

    We then prove the claim by induction on $k$. In the case $k=2$, we only have the case $e_1=e_2=2$, which falls into the case treated above.

    Assume now that $k>2$. We may also assume that $\max\{e_j\} \geq 3$, and we note that at most one $e_i$ equals $k$.  To ease notation we will assume that $\max\{e_j\}=e_1$, since the other cases are treated similarly.  Note that $n \geq 2$. We set
    \[
  e'_i    =\begin{cases}
    e_i-1, & \; \; \mbox{for} \;\; i=1,2, \\
    e_i, & \; \; \mbox{for} \;\; i>2.
    \end{cases}
\]

We have $\sum(e'_i-1)=2(k-2)$ and $e'_i \leq k-1$ for all $i$. If also $e_2 \geq 3$, then all $e'_i \geq 2$, and we may apply the induction hypothesis to find 
$\sigma'_1,\ldots,\sigma'_n \in \Sym(k-1)$ of orders $e'_1,\ldots,e'_n$ such that
    \begin{itemize}
    \item[(i)'] $\sigma'_1\cdots\sigma'_n=\id_{k-1}$,
    \item[(ii)'] the subgroup generated by $\sigma'_1,\ldots,\sigma'_n$ is transitive,
      \item[(iii)'] for all $i \in \{1,\ldots,n-1\}$, $\sigma'_i$ and $\sigma'_{i+1}$ are not disjoint. 
    \end{itemize}    
    Since $\sigma'_1$ and $\sigma'_2$ are not disjoint, there is one integer, say $x \in \{1,\ldots,k-1\}$, appearing in them both. We may write
    \[ \sigma'_1=(x \;\; a_1 \cdots a_s) \;\; \mbox{and} \;\; \sigma'_2=(b_1 \cdots b_t \;\; x).\]
    Set
\[  \sigma_1:=(x \;\; a_1 \cdots a_s\;\;k) \;\; \mbox{and} \;\; \sigma_2:=(k\;\;b_1 \cdots b_t \;\; x).\]  
Then, viewing $\sigma'_1$  and $\sigma'_2$ as elements of $\Sym(k)$, one may check that
\begin{equation} \label{eq:prodo}
  \sigma_1 \sigma_2= \sigma'_1 \sigma'_2.
  \end{equation}
Set now $\sigma_j:=\sigma'_j$ for all $j \in\{3,\ldots,n\}$ and view them as elements of $\Sym(k)$. Then, because of (i)'-(iii)' and \eqref{eq:prodo}, the cycles $\sigma_1,\ldots,\sigma_n \in \Sym(k)$ satisfy (i)-(iii) in the claim.

Assume now that $e_2=2$, so that $e'_2=1$.  We remark that $n>2$; indeed, if $n=2$, then $e_1=e-e_2=[2(k-1)+2]-2=2k-2$, which is incompatible with $e_1\leq k$ and $k>2$.  We have $2 \leq e'_i \leq k-1$ for all $i \neq 2$. Then we apply the  induction hypothesis to find 
$\sigma'_1,\sigma'_3,\ldots,\sigma'_n \in \Sym(k-1)$ of orders $e'_1,e'_3\ldots,e'_n$ such that
    \begin{itemize}
    \item[(i)''] $\sigma'_1\sigma'_3\cdots \sigma'_n=\id_{k-1}$,
    \item[(ii)''] the subgroup generated by $\sigma'_1,\sigma'_3,\ldots,\sigma'_n$ is transitive,
      \item[(iii)''] $\sigma'_1$ and $\sigma'_3$ are not disjoint, and for all $i \in \{3,\ldots,n-1\}$, $\sigma'_i$ and $\sigma'_{i+1}$ are not disjoint. 
    \end{itemize}    
    Since $\sigma'_1$ and $\sigma'_3$ are not disjoint, there is one integer, say $x \in \{1,\ldots,k-1\}$, appearing in them both. We may write
    \[ \sigma'_1=(x \;\; a_1 \cdots a_s).\]
    Set
\[  \sigma_1:=(x \;\; a_1 \cdots a_s\;\;k) \;\; \mbox{and} \;\; \sigma_2:=(k\;\;x).\]  
Then, viewing $\sigma'_1$  as element of $\Sym(k)$, one may check that
\begin{equation} \label{eq:prodo2} \sigma_1 \sigma_2= \sigma'_1.
   \end{equation}
Set now $\sigma_j:=\sigma'_j$ for all $j \in\{3,\ldots,n\}$ and view them as elements of $\Sym(k)$. Then, because of (i)''-(iii)'' and \eqref{eq:prodo2}, the cycles $\sigma_1,\ldots,\sigma_n \in \Sym(k)$ satisfy (i)-(iii) in the claim.  
\end{proof}
Having proved the claim, the lemma follows. 
  \end{proof}

  \begin{proof}[Proof of Theorem \ref{thm:main_mod}]
Since $\widetilde{\rho}(0,1,k;\e)=\widetilde{\rho}(g,1,k;\e)+g$, our assumption $\widetilde{\rho}(g,1,k;\e) \geq -g$, together with 
Lemma \ref{lemma:P1suP1}, yields that  the locus
 $G^1_k(\PP^1,(P_1,e_1),\ldots,(P_n,e_n))$ is nonempty and equidimensional of dimension $\widetilde{\rho}(g,1,k;\e)+g$ for general $P_1,\ldots,P_n \in \PP^1$.  This defines a subfamily  $\F_{k,(P_1,\ldots,P_n),\e}$ of $\F_k$ (cf. Definition \ref{def:Fk}) of dimension
 \begin{equation} \label{eq:dimFkn}
   \dim \F_{k,(P_1,\ldots,P_n),\e}=\widetilde{\rho}(g,1,k;\e)+g.
   \end{equation}

{\bf The case $\widetilde{\rho}(g,1,k;\e) \geq 0$.} For a general set of $g$ points $\{\xi_i=y_i+z_i\}_{1 \leq i \leq g}$ in $\Sym^2(\PP^1)$ the family of curves in $\F_{k,(P_1,\ldots,P_n),\e}$ passing through $\xi_1,\ldots,\xi_g$ has dimension 
$\widetilde{\rho}(g,1,k;\e)$ by \eqref{eq:dimFkn}. This family yields
the variety  $G^1_k(X,(P_1,e_1),\ldots,(P_n,e_n))$ on the $g$-nodal curve $X$ obtained by identifying the $g$ pairs $(y_i,z_i)$ of points on $\PP^1$. Thus, we have proved that $G^1_k(X,(P_1,e_1),\ldots,(P_n,e_n))$ has dimension $\widetilde{\rho}(g,1,k;\e)$.

  {\bf The case $\widetilde{\rho}(g,1,k;\e) < 0$.} Set $g':=\widetilde{\rho}(g,1,k;\e)+g <g$. For a general set of $g'$ points $\{\xi_i=y_i+z_i\}_{1 \leq i \leq g'}$ in $\Sym^2(\PP^1)$ the family of curves in $\F_{k,(P_1,\ldots,P_n),\e}$ passing through $\xi_1,\ldots,\xi_{g'}$ has dimension 
    $0$ by \eqref{eq:dimFkn}. Choose any set of distinct $g-g'$ points $\{\xi_i=y_i+z_i\}_{g'+1 \leq i \leq g}$ on any of the finitely many curves in  $\F_{k,(P_1,\ldots,P_n),\e}$ so obtained. Then, on the $g$-nodal curve $X$ obtained by identifying the $g$ pairs $(y_i,z_i)$ of points on $\PP^1$, the variety $G^1_k(X,(P_1,e_1),\ldots,(P_n,e_n))$ is $0$-dimensional.  

    Theorem \ref{thm:main_mod} now follows by semicontinuity, in view of \eqref{eq:hj2}, \eqref{eq:lambda} and    what we explained right after \eqref{eq:dimG1k}. 
   \end{proof}

 \begin{proof}[Proof of Theorem \ref{thm:mainHur}]    
   To prove the theorem we will prove that the general fibers of $\H_{g,k,\e} \to \M_g$ have dimension $\max\{0,n+\widetilde{\rho}(g,1,k;\e)\}$ on some component. 
We will prove this by degeneration to a $g$-nodal rational curve in $\overline{\M_g}$ as in the previous proof. 

With the notation therein, we set $\F_{k,\e}$ to be the  variety parametrizing
\[ \left(C_{\mathfrak{g}},P_1,\ldots,P_n\right) \;\; \mbox{such that} \;\;
  C_{\mathfrak{g}} \in \F_{k,(P_1,\ldots,P_n),\e} \]
for varying distinct  general  $P_1,\ldots,P_n$. Then $\F_{k,\e}$ is  equidimensional  of dimension
\begin{equation} \label{eq:dimFk}
  \dim \F_{k,\e}=\widetilde{\rho}(g,1,k;\e)+g+n.
\end{equation}

{\bf The case $n+\widetilde{\rho}(g,1,k;\e) \geq 0$.} The same argument as in the previous proof starting with $\F_{k,\e}$ shows that for a general set of $g$ points in $\Sym^2(\PP^1)$, the family of curves in $\F_{k,\e}$ passing through these points has dimension 
$\widetilde{\rho}(g,1,k;\e)+n \geq 0$ by \eqref{eq:dimFk}.
This produces, as in the previous proof, a $g$-nodal rational curve $X$ 
having a
$(\widetilde{\rho}(g,1,k;\e)+n)$-dimensional family of  
$g^1_k$s with the given ramification order at {\it some} $n$ points. Recalling \eqref{eq:hj2}, this proves that the fiber over $X$ of $\overline{\H_{g,k,\e}} \to \overline{\M_g}$ has dimension $n+\widetilde{\rho}(g,1,k;\e)$. The theorem follows by semicontinuity. 

{\bf The case $n+\widetilde{\rho}(g,1,k;\e) < 0$.} Set $g':=\widetilde{\rho}(g,1,k;\e)+g+n <g$. Then, for a general set of $g'$ points in $\Sym^2(\PP^1)$, the family of curves in $\F_{k,\e}$ passing through these points is zero-dimensional by \eqref{eq:dimFk}. Then, on the $g$-nodal curve $X$ obtained by identifying the $g'$ pairs of points on $\PP^1$ corresponding to the $g'$ points in $\Sym^2(\PP^1)$, plus an arbitrary set of $g-g'$ points, we have obtained a finite family of  
$g^1_k$s with the given ramification order at {\it some} $n$ points.
Recalling \eqref{eq:hj2}, this proves that the fiber over $X$ of $\overline{\H_{g,k,\e}} \to \overline{\M_g}$ has dimension $0$. The theorem follows by semicontinuity.
\end{proof}

\section{Building on the construction and results  in \cite{CK}}\label{sec:rew}

To prove Theorem \ref{thm:main}, we will use a degeneration argument developed in \cite{CK}. 

Let $p\geq 3$.  We denote by 
$\K_p$ the \emph{moduli space} of primitively polarized $K3$ surfaces of genus $p$,
which is irreducible of dimension $19$.  The following result is central in our investigation:

\begin{thm}\label{thm:mainCK}  If $(S,H)$ is a general primitive $K3$ surface of genus $p$, then
$V^k_{|H|,\delta} \neq \emptyset$ if and only if \eqref{eq:boundintro} is satisfied. When nonempty, $V^k_{|H|,\delta}$ is equidimensional of dimension $\min\{2(k-1),g\}$.
\end{thm}

The nonemptiness part of the theorem is \cite[Thm. 0.1]{CK}, whereas the assertion about equidimensionality follows from  \cite[Rem. 5.6]{KLM}.
In this section we review and build on the proof of the assertion concerning nonemptiness in order to approach the refined ramified version in \S \ref{sec:proofs}. The general strategy consists in degenerating the polarized K3 surface $(S,H)$ to a pair $(S_0,H_0)$  (\S \ref{subsec:limitk3}) so that the Severi variety $V_{|H|,\delta}$ degenerates to a locus containing an irreducible component $V(\alpha_{1}, \ldots,  \alpha_{p})$ (cf. Definition \ref{def:Valpha}), whose curves all contain a special component isomorphic to $\bP^1$ with a collection of pairs of distinguished points (\S \ref{subsec:limitV}). The computation of the dimension of $V^k_{|H|,\delta}$ reduces in this way to computing the dimension of families of  $g^1_k$s on $\bP^1$ satisfying certain compatibility conditions. More precisely, each pair of distinguished points must belong to a divisor of the given $g^1_k$. Such a computation is obtained by intersecting appropriate curves in $\bP^2$ (\S \ref{subsec:Vk}).  

\subsection{Limit K3 surface}\label{subsec:limitk3} We describe now the degeneration of $(S,H)\in \K_p$ to a pair $(S_0,H_0)$ (cf. \cite[\S 4]{CK}).

Let $E$ be a smooth elliptic curve with 
two distinct {\it general} line bundles $L_1,L_2 \in \Pic^2(E)$. Consider the embedding of $E \subset \PP^{2p-1}$ as  a smooth elliptic normal curve of degree $2p$ given by $L_2^ {\otimes p}$. Let $R_1$ and $R_2$ be the rational normal scrolls of degree $2p-2$ in
$\PP^{2p-1}$ defined by $L_1$ and $L_2$, respectively, that is, $R_i$ is the union of lines
spanned by the divisors of  $\vert L_i\vert$. Then
\begin{itemize}
\item $R_1 \cong \PP^1 \x \PP^1$ and $R_2 \cong \mathbb{F}_2$,
\item  $R_1 \cap R_2 =E$ transversally,
\item $E$ is anticanonical on each $R_i$.  
\end{itemize}
Set $R:=R_1 \cup R_2$.

 We will denote by $f:\PP^1 \x \PP^1 \to \PP^1$ the projection to $\PP^1$ defined by the pencil $|L_1|$. Let $\mathfrak{s}_i$ and $\mathfrak{f}_i$ denote the classes of the nonpositive section and fiber, respectively, of $R_i$, for $1\le i\le 2$.  Thus, the members of $|\mathfrak{f}_1|$ are the fibers of $f$.  
Then $\O_{R_1}(1) \cong \O_{R_1}(\mathfrak{s}_1 +(p-1)\mathfrak{f}_1)$ and 
$\O_{R_2}(1) \cong \O_{R_2}(\mathfrak{s}_2 +p\mathfrak{f}_2)$. The section $\mathfrak{s}_2$ does not intersect $E$, hence lies in the smooth locus of $R$. In particular, $\mathfrak{s}_2$ is a Cartier divisor on $R$, so that 
\begin{equation}
  \label{eq:classnuova}
H_0:  =\O_{R}(1)\otimes \O_R(-\mathfrak{s}_2)
\end{equation}
is also Cartier, with $H_0^2=2p-2$ and $\mathfrak{s}_2 \cdot H_0=p$. By
\cite[Lem.  4.1]{CK} the pair $(S_0:=R,H_0)$ can be flatly deformed to a member of $\K_p$. More precisely, there is a flat family $\pi:\X\to \DD$, where $\DD$ is the disc, such that the special fiber
$\pi^{-1}(0)\simeq R$ and the fibers $S_t:=\pi^{-1}(t)$ for $t \neq 0$ are smooth $K3$ surfaces, and a  line bundle $\H$ on $\X$ such that $\H|_{S_0}=H_0$ and, setting
$H_t:=\H|_{S_t}$ also for $t \neq 0$, we have that $(S_t,H_t) \in \K_p$ is general.

Note that
the total space $\X$ has $16$ double points on $E$ in the central fibre, and is otherwise smooth. One may perform a small resolution of the singularities of $\X$, obtaining a new family with unchanged fibers for $t \neq 0$, but whose new central fiber is a birational modification of $R$. Precisely one can make sure that  the new central fiber is $R_1\cup \widetilde R_2$, where $\widetilde R_2$ is  the blow up of  
$R_2$ at the $16$ singular points of $\X$, and $R_1$ and $\widetilde R_2$ meet transversally along $E\subset R_1$ and its strict transform on $\widetilde R_2$. 
Since we can make sure that the interesting curves we work with on $R$ will lie off the singular points of $\X$, we can choose to work with $\X$ without caring about the singularities. 

\subsection{Limit Severi varieties }\label{subsec:limitV} We now describe limit curves of $V_{|H_t|,\delta}$, following \cite[\S 6]{CK}.

Let $m$ be a positive integer satisfying $m \leq p$.
A {\it chain of length $2m-1$} is a
sum of $2m-1$ distinct lines
\[ f_{2,1}+f_{1,1}+f_{2,2}+f_{1,2} + \cdots + f_{2,m-1}+f_{1,m-1}+f_{2,m}, \; f_{i,j} \in |\mathfrak{f}_i|, \]
where 
$f_{1,j}$ intersects only $f_{2,j}$ and $f_{2,j+1}$, for $j=1, \ldots, m-1$.
The chain intersects $E$ in $2m$ points, consisting of $m$ divisors of  $|L_2|$, and 
$2m-2$ of them lie on the intersections between two lines, whereas the remaining two are on $f_{2,1}$ and $f_{2,m}$ and will be denoted by $a$ and $b$, respectively. This pair of points will be called {\it the distinguished pair of points of the chain}. Figure \ref{fig:1}  shows a chain of length $9$ and its distinguished pair of points. If the chain is contained in a member $C$ of $|H_0|$, and we denote by $\Gamma$ the section of the ruling on $R_1$ contained  in $C$, the distinguished points are $a=\Gamma \cap f_{2,1}$ and $b=\Gamma \cap f_{2,m}$.
Denoting the restriction of $f$ to $E$ still by $f$ (so that $f:E \to \Gamma \cong \PP^1$ is the morphism determined by the linear series $|L_1|$), we  note that $a+b \in f_*\left(|mL_2-(m-1)L_1|\right)$. 

\begin{figure}
\begin{center}
\begin{tikzpicture}[>=stealth,scale=0.8]
 
\draw[thick,blue] 
 (0.5,-1) to[out=100,in=-70]
(-0.8,0.6)  to[out=110,in=-90]
(-0.9,1)  to[out=90,in=-90]
(-1,2)  to[out=90,in=-90]
(-1,3)  to[out=90,in=-110]
(-1,4.7) to[out=70,in=-45]
(1,3.7);

\draw[blue] (0.2,4.5) node[above right] {$\Gamma$};

\draw (-0.7,0) -- (1.7,0) node[right] {$f_{2,5}$};
\draw (-0.7,1) -- (1.7,1) node[right] {$f_{2,4}$};
\draw (-0.7,2) -- (1.7,2) node[right] {$f_{2,3}$};
\draw (-0.7,3) -- (1.7,3) node[right] {$f_{2,2}$};
\draw (-0.7,4) -- (1.7,4) node[right] {$f_{2,1}$};

\draw (1.5,-0.5) -- (-1.4,1.9) node[left] {$f_{1,4}$};
\draw (1.5,0.5) -- (-1.4,2.9) node[left] {$f_{1,3}$};
\draw (1.5,1.5) -- (-1.4,3.9) node[left] {$f_{1,2}$};
\draw (1.5,2.5) -- (-1.4,4.9) node[left] {$f_{1,1}$};

\draw[blue] (-0.95,1.543) circle [radius=2.5pt];
\draw[blue] (-1,2.56) circle [radius=2.5pt];
\draw[blue] (-1.05,3.6) circle [radius=2.5pt];
\draw[blue] (-1.04,4.6) circle [radius=2.5pt];

\draw[red,fill] (-0.35,0) circle [radius=1.5pt] node[below left] {$b$};
\draw[red,fill] (0.67,4) circle [radius=1.5pt] node[above right] {$a$};
\end{tikzpicture}

\caption{A chain of length $9$ contained in a member of $|H_0|$}
\label{fig:1}
\end{center}
\end{figure}
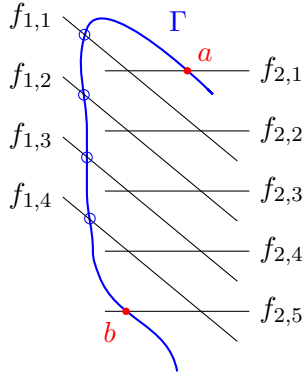

We denote by $\C_{m}$ the family of  chains of length $2m-1$, which is a locally closed subvariety of a Hilbert scheme of curves on $R$. In \cite[Lem. 4.2]{CK} it is proved that the map
\begin{equation} \label{eq:chain-pair}
  \C_m \longrightarrow |mL_2-(m-1)L_{1}|
\end{equation}
sending a  chain of length $2m-1$ to 
its pair of distinguished points on $E$ is  birational and  injective. 

\begin{defn} \label{def:Valpha}
  For each tuple $(\alpha_1,\ldots, \alpha_p)$ of {\it nonnegative} integers satisfying\begin{equation}
  \label{eq:condex2}
  p=\displaystyle\sum_{j=1}^p j\alpha_j
\end{equation}
one defines  $V(\alpha_{1}, \ldots,  \alpha_{p})$ to be the locally closed subset of reduced curves in $|H_0|$ not passing through the singular locus of $\X$ and containing exactly
$\alpha_{j}$  chains of length $2j-1$, for $j=1,\ldots,p$.
\end{defn}

By \eqref{eq:chain-pair}
we have a morphism
\begin{equation} \label{eq:alphas}
\xymatrix{
\nu: V(\alpha_{1}, \ldots,  \alpha_{p}) \ar[r] &  
\displaystyle\prod_{j=1}^p \Sym^{\alpha_j}|jL_2-(j-1)L_{1}|,
} 
\end{equation}
which is proved to be bijective in \cite[Proof of Prop. 5.2]{CK}. Indeed, its inverse is given in the following way: the coordinates of any member of the target uniquely defines a chain, and the intersection between the union of the lines in $|\f_2|$ of these chains and $E$ is an effective divisor lying in $|(\sum j \alpha_j)L_2|=|pL_2|=|\O_E(1)|$, by \eqref{eq:condex2}. This can be extended to a unique hyperplane section of $R$, which necessarily contains $\s_2$, so it gives a curve in $V(\alpha_{1}, \ldots,  \alpha_{p})$.
As the dimension of the target in \eqref{eq:alphas} is
\begin{equation}
  \label{eq:genere0}
  g:=\sum_{j=1}^p\alpha_j,
\end{equation}
we see that
$V(\alpha_{1}, \ldots,  \alpha_{p})$
is $g$-dimensional.

Any curve $C$ in $V(\alpha_{1}, \ldots,  \alpha_{p})$ is nodal, and we  denote as above by $\Gamma$ the section of the ruling on $R_1$ contained  in $C$ (i.e., $C$ is the union of $\Gamma$ and of  chains).
We will call $\Gamma$ the {\it sectional component of $C$}. 
The curve $C$ has
\begin{equation}
  \label{eq:numeronodi}
  \delta=\delta(\alpha_{1}, \ldots,  \alpha_{p}):= 
\displaystyle\sum_{j=1}^p (j-1) \alpha_{j} 
\end{equation}
nodes lying off $E$, which we call the 
{\it marked nodes} of $C$, and they all lie on  its sectional component $\Gamma$. In Figure \ref{fig:1} they are marked with circles. We note that \eqref{eq:condex2}, \eqref{eq:genere0} and \eqref{eq:numeronodi} yield 
\begin{equation}
  \label{eq:genere}
  g=p-\delta.
\end{equation}
It is proved in \cite[Prop. 5.2]{CK} that $V(\alpha_{1}, \ldots,  \alpha_{p})$
is a component of the flat limit of $V_{| H_t|,\delta}$.

\subsection{Limit Severi varieties of curves whose normalization carry a $g^1_k$}\label{subsec:Vk}
We now describe limit curves of $V^k_{|H_t|,\delta}$ and compute the dimension of the component they define, following \cite[\S 6]{CK}.

As in Figure \ref{fig:2}, we denote by $\widetilde{C}$ the partial normalization of $C$ at its $\delta$ marked nodes and by $C'$ the stable model of $\widetilde{C}$, obtained by contracting all chains of $C$. 
We already noted  that  the distinguished pair of points of each chain contained in $C$ lies on $\Gamma$. There are $g$ such pairs.  
Thus, $C'$ is equivalently the $g$-nodal curve obtained by identifying each pair of distinguished pair of points on $\Gamma$, as shown in Figure \ref{fig:2}. We note that the natural map $\Gamma \to C'$ is the normalization, and that a degree $k$ cover $\Gamma \to \PP^1$ factors through $C'$ if and only if all
distinguished pairs of points lie in the same fibers of $\Gamma \to \PP^1$. This motivates the following definition.  

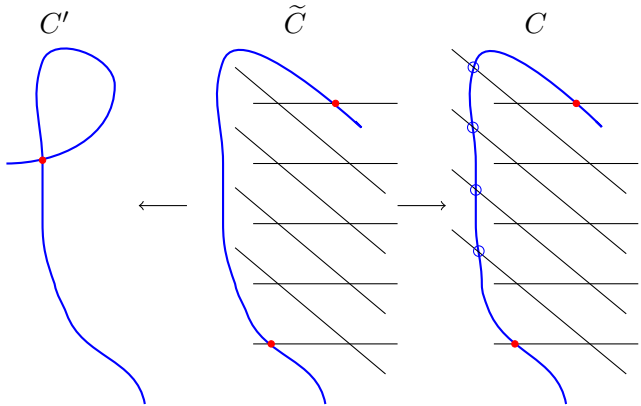
\begin{figure} 
\begin{center}
\begin{tikzpicture}[>=stealth,scale=0.8]
\draw[thick,blue] 
 (0.5,-1) to[out=100,in=-70]
(-0.8,0.6)  to[out=110,in=-90]
(-0.9,1)  to[out=90,in=-90]
(-1,2)  to[out=90,in=-90]
(-1,3)  to[out=90,in=-110]
(-1,4.7) to[out=70,in=-45]
(1,3.7);

\draw (0,5) node[above] {$C$};

\draw (-0.7,0) -- (1.7,0);
\draw (-0.7,1) -- (1.7,1);
\draw (-0.7,2) -- (1.7,2);
\draw (-0.7,3) -- (1.7,3);
\draw (-0.7,4) -- (1.7,4);

\draw (1.5,-0.5) -- (-1.4,1.9);
\draw (1.5,0.5) -- (-1.4,2.9);
\draw (1.5,1.5) -- (-1.4,3.9);
\draw (1.5,2.5) -- (-1.4,4.9);

\draw[blue] (-0.95,1.543) circle [radius=2.5pt];
\draw[blue] (-1,2.56) circle [radius=2.5pt];
\draw[blue] (-1.05,3.6) circle [radius=2.5pt];
\draw[blue] (-1.04,4.6) circle [radius=2.5pt];

\draw[red,fill] (-0.35,0) circle [radius=1.5pt];
\draw[red,fill] (0.67,4) circle [radius=1.5pt];

\draw[-to]  (-2.3,2.3) -- (-1.5,2.3);

%%%%%%%%%%%%%
\draw[thick,blue] 
 (-3.5,-1) to[out=100,in=-70]
(-4.8,0.5)  to[out=110,in=-80]
(-5,1)  to[out=110,in=-90]
(-5.2,2)  to[out=90,in=-90]
(-5.2,3)  to[out=90,in=-110]
(-5.2,4.7) to[out=70,in=-45]
(-3,3.7);

\draw (-4,5) node[above] {$\widetilde{C}$};

\draw (-4.7,0) -- (-2.3,0);
\draw (-4.7,1) -- (-2.3,1);
\draw (-4.7,2) -- (-2.3,2);
\draw (-4.7,3) -- (-2.3,3);
\draw (-4.7,4) -- (-2.3,4);

\draw (-2.5,-0.5) -- (-5,1.6);
\draw (-2.5,0.5) -- (-5,2.6);
\draw (-2.5,1.5) -- (-5,3.6);
\draw (-2.5,2.5) -- (-5,4.6);

\draw[red,fill] (-4.4,0) circle [radius=1.5pt];
\draw[red,fill] (-3.33,4) circle [radius=1.5pt];

\draw[to-]  (-6.6,2.3) -- (-5.8,2.3);

%%%%%%%%%%%%

\draw (-8,5) node[above] {$C'$};
\draw[thick,blue] 
 (-6.5,-1) to[out=100,in=-70]
(-7.8,0.5)  to[out=110,in=-80]
(-8,1)  to[out=110,in=-90]
(-8.2,2)  to[out=90,in=-90]
(-8.2,3)  to[out=90,in=-110]
(-8.2,4.7) to[out=70,in=90]
(-7,4.3) to[out=-90,in=0]
(-8.8,3)
;

\draw[red,fill] (-8.2,3.055) circle [radius=1.5pt];

\end{tikzpicture}

\caption{The partial normalization $\widetilde{C}$ of $C$ and its stable model $C'$, around a chain}\label{fig:2}
\end{center}

\end{figure}

\begin{defn} \label{def:base}
  Let $C \in V(\alpha_1,\ldots,\alpha_p)$.   We say that a $g^1_k$ on the sectional component $\Gamma$ of $C$ is a {\it descending $g^1_k$} if all distinguished pairs of points of each chain contained in $C$ belong to divisors of the $g^1_{k}$.

  We denote by $G^1_k(C')$ the variety of all descending $g^1_k$s on $\Gamma$ and by $\G^k(\alpha_1,\ldots,\alpha_p)$ the variety parametrizing pairs $(C,\mathfrak{g})$ such that $C \in V(\alpha_1,\ldots,\alpha_p)$ and $\mathfrak{g} \in G^1_k(C')$. 

  We denote by $V^k(\alpha_1,\ldots,\alpha_p)$ the image of the forgetful map
  \[ \sigma: \G^k(\alpha_1,\ldots,\alpha_p) \longrightarrow V(\alpha_1,\ldots,\alpha_p) \]
  (which has fibers
  $G^1_k(C')$).  
\end{defn}

By what we said prior to the definition, $C$ lies in $V^k(\alpha_1,\ldots,\alpha_p)$ if and only if $C'$ lies in $\overline{\M^1_{g,k}}$ (see \cite[Lem. 5.3]{CK}).

In \cite[Prop. 6.5 and Cor. 6.6]{CK} the following is proved:
\begin{prop} \label{prop:CK6.5}
  If   \eqref{eq:condex2} and
\begin{equation}
\label{eq:condex3} \alpha_{j} \leq 2(k-1) \; \; \mbox{for all} \; \; j 
\end{equation}
hold, the variety 
$V^k( \alpha_1,\ldots,\alpha_p)$
is nonempty and irreducible of dimension $\min\{2(k-1),g\}$, and is contained in the limit of
$V^k_{|H_t|,\delta}$ as $t$ tends to $0$. 
\end{prop} 

It is also proved in \cite{CK} that there are $(\alpha_1,\ldots,\alpha_p)$ satisfying conditions \eqref{eq:condex2} and \eqref{eq:condex3} whenever \eqref{eq:boundintro} is satisfied (cf. \cite[Prop. 7.1]{CK}).

 To prove Proposition \ref{prop:CK6.5}, the variety 
$V^k( \alpha_1,\ldots,\alpha_p)$ is explicitly constructed in the following way.   
Let $f:  E \rightarrow  \PP^1$ be the morphism determined  by  $\vert L_1\vert$
and  $f^{(2)}: \Sym^2(E) \rightarrow \Sym^2(\PP^1)$ the induced  map.
We identify $\Sym^2(\PP^1)$ with  $\PP^2$ as in the previous section.
We denote by  $\mathfrak{c}_{j}$, for $1\le j \le n$, the image via $f^{(2)}$ of the smooth rational curves in $\Sym^2(E)$ defined by the 
pencils $|jL_{2}-(j-1)L_{1}|$. These are distinct conics, each intersecting the diagonal $\Delta$ in four distinct points corresponding to the ramification points of the pencils. Recall from the previous section that to any each $\mathfrak{g} \in G^1_k(\PP^1)$ we associate the degree $k-1$ curve in $\PP^ 2=\Sym^2(\PP^1)$ given by $ C_{\mathfrak{g}}= \{ W \in \Sym^2(\PP^1) \; | \; \mathfrak{g}(-W) \geq 0\}$
and that we denote the family of such curves by $\F_k$ (cf. Definition \ref{def:Fk}).

We state here a stronger version of \cite[Lem. 6.3]{CK}, as we will need it later to prove our main results:

\begin{lemma} \label{lemma:sobs}
  Fix a reduced  curve $C_{\mathfrak{g}}$ in $\F_k$. Then, for general $L_2 \in \Pic^2(E)$, the curve 
  $C_{\mathfrak{g}}$ intersects each $\mathfrak{c}_{j}$ transversally in $2(k-1)$ distinct points. Moreover, varying $C_{\mathfrak{g}}$ in $\F_k$, none of these points  are fixed.
\end{lemma}

\begin{proof}
  The fibers of the Albanese (summation) map $ \mbox{alb}:\Sym^2(E) \to E$ are the one-dimensional family of $g^1_2$s on $E$, and each curve $\mathfrak{c}_{j}$ is the image of one such fiber. Varying $L_2 \in  \Pic^2(E)$ the fibers sent to $\mathfrak{c}_{j}$ cover all fibers of $\mbox{alb}$, so that $\mathfrak{c}_{j}$ becomes the image of a general fiber.  Each pair of fibers of $a$ moves in a base point free linear system, whence $(f^{(2)})^{-1}C_{\mathfrak{g}}$ has transversal intersection with the general fiber of $\mbox{alb}$ and the intersection occurs outside of the ramification locus of $f^{(2)}$. This proves that also the intersection $C_{\mathfrak{g}} \cap \mathfrak{c}_{j}$ is transversal.

  The case $L_2=L_1$ yields $\mathfrak{c}_j=\Delta$ for all $j$. Varying $\mathfrak{g}$, none of the points in $C_{\mathfrak{g}} \cap \Delta$ are fixed, proving the last assertion also for general $L_2$. 
\end{proof}

By   \eqref{eq:condex3}  we can pick  $\alpha_{j}$ of the intersection points $C_{\mathfrak{g}} \cap \mathfrak{c}_{j}$. Each of these $\alpha_{j}$ intersection points gives rise to a chain of length $2j-1$ with distinguished pair given by the
$2$-cycle corresponding to the point itself (recall \eqref{eq:chain-pair}) and the one-to-one correspondence \eqref{eq:alphas} yields the existence of curves in 
$V( \alpha_1,\ldots,\alpha_p)$ containing all the above chains, which 
lie in $V^k( \alpha_1,\ldots,\alpha_p)$ by construction.
In other words, we have natural surjective maps
\begin{equation} \label{eq:natmaps}
  \xymatrix{
    \G^k( \alpha_1,\ldots,\alpha_p) \ar[r]^{\sigma_k} \ar[d]_{\gamma_k}& V^k( \alpha_1,\ldots,\alpha_p) \\
    \F_k, &
  }
  \end{equation}
where $\sigma_k$ is the forgetful map and $\gamma_k$ is finite. The fiber of $\gamma_k$ over a curve $C_{\mathfrak{g}}$ is given by all the choices of $\alpha_j$ of the intersection points in each $C_{\mathfrak{g}} \cap \mathfrak{c}_{j}$.

When $g=\sum \alpha_i >2(k-1)$, this is used in \cite{CK} to prove that $V^k( \alpha_1,\ldots,\alpha_p)$ is irreducible (and nonempty) of dimension $2(k-1)$ and that its general element corresponds to finitely many of the curves $C_{\mathfrak{g}}$. When $g \leq 2(k-1)$, one has $V^k( \alpha_1,\ldots,\alpha_p)=
V( \alpha_1,\ldots,\alpha_p)$ and its general element corresponds to a family of dimension $2(k-1)-\sum\alpha_{j}=2(k-1)-g=\rho(g,1,k)$ of the curves $C_{\mathfrak{g}}$. This also shows that the generic fiber dimension of $\sigma_k$ is $\max\{0,\rho(g,1,k)\}$, in other words that  $\dim G^1_k(C')=\max\{0,\rho(g,1,k)\}$ 
for general $C$ in $V^k( \alpha_1,\ldots,\alpha_p)$. 

We will also make use of the following observation, implicitly used in \cite[Proof of Cor. 6.6(iv)]{CK}:

\begin{lemma} \label{lemma:NN}
If $\mathfrak{g}$ is base point free, then any corresponding curve in $V^k( \alpha_1,\ldots,\alpha_p)$ has the property that the stable model of a partial normalization of $C$ at fewer than its $\delta$ marked nodes does not lie in $\overline{\M^1_{g,k}}$.
\end{lemma}

\begin{proof}
Such a partial normalization would need to have at least one more component other than $\Gamma$,  intersecting $\Gamma$ in three points. For the stable model of such a curve to lie in $\overline{\M^1_{g,k}}$, the pencil $\mathfrak{g}$ on $\Gamma$ would need to have base points, since it would have to induce a map from $\Gamma$ to $\PP^1$ of degree $<k$. 
\end{proof}

 Under conditions \eqref{eq:condex2} and \eqref{eq:condex3} we thus have a relative scheme
$\pi_{\delta}: \V_{\delta} \to \DD$ of dimension $\dim \V_{\delta}=g+1$ such that the fiber over $t \neq 0$ is a component of  $V_{|H_t|,\delta}$ (of dimension $g:=p-\delta$) and a subscheme 
$\V^k_{\delta} \subset \V_{\delta}$ such that the fiber over $t \neq 0$ is a component of  $V^k_{|H_t|,\delta}$ of dimension $\min\{2(k-1),g\}$ and the special fiber contains a dense open subset of $V^k( \alpha_1,\ldots,\alpha_p)$. By construction, the map from $\V^k_{\delta}$ sending a curve to the stable model of its normalization at the $\delta$ marked nodes lands in $\overline{\M^1_{g,k}}$,  the compactification of $\M^1_{g,k}$ in the Deligne-Mumford compactification $\overline{\M_g}$ of $\M_g$;  in other words, we have moduli maps
\begin{equation} \label{eq:dagger}
  \xymatrix{
  \V_{\delta}^k  \ar[d] & \subset & \V_{\delta}  \ar[d]^{\sigma_{\delta}}  \\
  \overline{\M^1_{g,k}} & \subset & \overline{\M_{g}}. }
\end{equation}

\section{Proof of Theorem \ref{thm:main}} \label{sec:proofs}

The theorem is trivial in the case $p=2$, so we will assume that $p\geq 3$.

We will prove Theorem \ref{thm:main} by degeneration, following the structure of the proof of Theorem \ref{thm:mainCK} as reviewed in \S \ref{sec:rew}. We will use the same degeneration of the K3 surface $(S,H)$ to $(S_0,H_0)$, and of the Severi variety $V_{|H|,\delta}$ to a locus containing an irreducible component $V(\alpha_{1}, \ldots,  \alpha_{p})$. We will introduce a pointed version $V_n(\alpha_1,\ldots,\alpha_p)$ of $V(\alpha_{1}, \ldots,  \alpha_{p})$, in order to cut out subloci $V^k_{n,\e}(\alpha_1,\ldots,\alpha_p)$ by imposing fixed ramification at the marked points. We will study these subloci by reducing the problem to the study of $g^1_k$s on $\bP^1$ with prescribed ramification and with the same compatibility conditions as in \S \ref{sec:rew}, that is, distinguished pairs of points must belong to a divisor of the $g^1_k$.

Recall \eqref{eq:dagger}. 
Denoting by $\V_{\delta,n}$ the locus of $n$-pointed curves $(C,x_1,\ldots, x_n)$, with $[C] \in \V_{\delta}$,  where the  marked points are disjoint from the nodes of the curves, we have $\dim \V_{\delta,n}=g+n+1$ and a relative scheme
\[ \pi_{\delta,n}: \V_{\delta,n}  \longrightarrow  \DD\]
and 
moduli map
\[ \sigma_{\delta,n}: \V_{\delta,n} \longrightarrow  \overline{\M_{g,n}}.\]
We define $\V^k_{\delta,\e}:=\sigma_{\delta,n}^{-1}(\overline{\M_{g,k,\e}})$.

To prove the first assertion in Theorem \ref{thm:main}, we would like to prove that $\V^k_{\delta,\e} \cap \pi_{\delta,n}^{-1}(t)$ is nonempty and has a component of the expected codimension for general $t$. Since $\M_{g,k,\e}$ has a component of the expected codimension by Theorem \ref{thm:main_mod}, it suffices to prove that $\V^k_{\delta,\e} \cap \pi_{\delta,n}^{-1}(0)$ is nonempty and has a component of the expected codimension. We will find such a component in Proposition \ref{prop:bordo} below.

\begin{defn} \label{def:lughi}
  We define $V_n(\alpha_1,\ldots,\alpha_p)$ to be the variety parametrizing $(C,x_1,\ldots,x_n)$ such that $C \in V( \alpha_1,\ldots,\alpha_p)$ and $x_1,\ldots, x_n$ are distinct points on the sectional component $\Gamma$  of $C$ different from the nodes of $C$. 
  
We denote by $G^1_k(C',(x_1,e_1),\ldots,(x_n,e_n))$ the variety parametrizing $\mathfrak{g} \in G^1_k(C')$ (i.e.
$\mathfrak{g}$ is a descending $g^1_k$ on $\Gamma$) with ramification of order $e_i \geq 2$ at each $x_i$.

We define $\G^k_{n,\e}( \alpha_1,\ldots,\alpha_p)$ to be the variety parametrizing  pairs $\Big(\left(C,x_1,\ldots,x_n\right),\mathfrak{g}\Big)$ such that $(C,x_1,\ldots,x_n) \in V_{n,\e}( \alpha_1,\ldots,\alpha_p)$ and $\mathfrak{g} \in G^1_k(C',(x_1,e_1),\ldots,(x_n,e_n))$. 

We define  $V^k_{n,\e}(\alpha_1,\ldots,\alpha_p)$ to be the image of the forgetful map
\[ \sigma_{k,n,\e}: \G^k_{n,\e}( \alpha_1,\ldots,\alpha_p) \to V_{n}(\alpha_1,\ldots,\alpha_p), \]
which has fibers $G^1_k(C',(x_1,e_1),\ldots,(x_n,e_n))$.  
\end{defn}

\begin{prop} \label{prop:bordo}
  Assume that \eqref{eq:condex2} and \eqref{eq:condex3} hold, as well as
  $\widetilde{\rho}(g,1,k;\e) \geq -g$.
Then $V^k_{n,\e}(\alpha_1,\ldots,\alpha_p)$ is equidimensional of codimension $\max\{0,-\widetilde{\rho}\}$.

More precisely,
\begin{itemize}
\item[(i)] if $\widetilde{\rho}\geq 0$, then $V^k_{n,\e}(\alpha_1,\ldots,\alpha_p)=V_{n}(\alpha_1,\ldots,\alpha_p)$ and for the general element $(C,x_1,\ldots,x_n) \in V_{n}(\alpha_1,\ldots,\alpha_p)$ the variety $G^1_k(C',(x_1,e_1),\ldots,(x_n,e_n))$ has dimension $\widetilde{\rho}$ and the general member in any component satisfies $(\star)$ (cf. Notation \ref{not:star});
\item[(ii)] if $\widetilde{\rho}< 0$, then $V^k_{n,\e}(\alpha_1,\ldots,\alpha_p)$ is a proper subvariety of $V_{n}(\alpha_1,\ldots,\alpha_p)$ of dimension
  $g+n+\widetilde{\rho}$, and for general $(C,x_1,\ldots,x_n)$  in any component of $V^k_{n,\e}(\alpha_1,\ldots,\alpha_p)$  the variety $G^1_k(C',(x_1,e_1),\ldots,(x_n,e_n))$ is $0$-dimensional and any member satisfies $(\star)$;  moreover,
  \begin{itemize}
  \item[(ii-a)] if $n+\widetilde{\rho}\geq 0$, then $V^k_{n,\e}(\alpha_1,\ldots,\alpha_p)$ dominates $V(\alpha_1,\ldots,\alpha_p)$ via the forgetful map, and
  \item[(ii-b)] if $n+\widetilde{\rho}< 0$, then the forgetful map is generically finite.
  \end{itemize}
\end{itemize}
\end{prop}

\begin{proof}
  We follow the lines of the proofs of Theorems \ref{thm:main_mod} and \ref{thm:mainHur}. In particular, we let $\F_{k,(P_1,\ldots,P_n),\e}$ and $\F_{k,\e}$  be the subfamilies of $\F_k$ as defined therein. In those proofs, we considered the subfamilies of curves passing through sets of {\it general points} $\xi_i$ in $\Sym^2(\PP^1)$. Now we will need those points to lie in $\mathfrak{c}_1 \cup \cdots \cup \mathfrak{c}_p$.

We recall that the assumption $\widetilde{\rho} \geq -g$ along with
Lemma \ref{lemma:P1suP1} yields that the locus
 $G^1_k(\PP^1,(P_1,e_1),\ldots,(P_n,e_n))$ is nonempty and equidimensional of dimension $\widetilde{\rho}+g$ for general $P_1,\ldots,P_n \in \PP^1$.  We also need to remark that the general member in any component of the variety 
  $G^1_k(\PP^1,(P_1,e_1),\ldots,(P_n,e_n))$ satisfies $(\star)$ (cf. Notation \ref{not:star} and Remark \ref{rem:main_mod}). In particular, the general member in any component of $\F_{k,(P_1,\ldots,P_n),\e}$ and $\F_{k,\e}$ corresponds to a $g^1_k$ satisfying $(\star)$.
  
Since none of the curves $\mathfrak{c}_{j}$ is contained in a $C_{\mathfrak g}$ by Lemmas \ref{lemma:Cgred} and \ref{lemma:sobs}, the passing through any set of general points on $\mathfrak{c}_1 \cup \cdots \cup \mathfrak{c}_p$ imposes independent conditions on the families $\F_{k,(P_1,\ldots,P_n),\e}$ and $\F_{k,\e}$
and the general members in  each component of  the resulting families intersect each $\mathfrak{c}_j$ transversally in $2(k-1)$ distinct points.

We
have surjective maps that are the marked versions of the ones in \eqref{eq:natmaps}:
\begin{equation} \label{eq:natmapsM}
  \xymatrix{
    \G^k_{n,\e}( \alpha_1,\ldots,\alpha_p) \ar[r]^{\sigma_{k,n,\e}} \ar[d]_{\gamma_{k,n,\e}}& V^k_{n,\e}( \alpha_1,\ldots,\alpha_p) \\
    \F_{k,\e}, &
  }
  \end{equation}
  where, again, $\sigma_{k,n,\e}$ is the forgetful map (forgetting the $g^1_k$) and
  $\gamma_{k,n,\e}$ is finite.

{\bf The case $\widetilde{\rho} \geq 0$.} For a general set of $\alpha_j$ points on each $\mathfrak{c}_j$, the family of curves in $\F_{k,(P_1,\ldots,P_n),\e}$ passing through these points has dimension 
\[
\dim \F_{k,(P_1,\ldots,P_n),\e}- \sum \alpha_j \stackrel{\eqref{eq:dimFkn}}{=}
  \widetilde{\rho}+g-\sum \alpha_j \stackrel{\eqref{eq:genere}}{=} \widetilde{\rho}
\]
and the  general member in any component corresponds to a $g^1_k$ satisfying $(\star)$. This family produces, by the bijection \eqref{eq:alphas}, the same curve $C$ in $V(\alpha_1,\ldots,\alpha_p)$, which lies in $V^k_{n,\e}(\alpha_1,\ldots,\alpha_p)$, with the marked points of ramification being $P_1,\ldots,P_n \in \Gamma \subset C$. This means that the fiber dimension of
$\sigma_{k,n,\e}$ is $\widetilde{\rho}$. 
The total family of pointed curves produced is therefore  of dimension 
\[ \dim \F_{k,\e}-\widetilde{\rho}\stackrel{\eqref{eq:dimFk}}{=} \widetilde{\rho}+g+n-\widetilde{\rho}=g+n,\]
whence it covers all of $V_n(\alpha_1,\ldots,\alpha_p)$. This proves (i). 

{\bf The case $\widetilde{\rho} < 0$.} For any $p$-tuple of nonnegative integers $(\alpha'_1,\ldots,\alpha'_p)$ such that
$\alpha'_i \leq \alpha_i$ for all $i$ and $\sum \alpha'_i=\widetilde{\rho}+g$, and any general choice of $\alpha'_j$ points on each $\mathfrak{c}_j$, the family of curves in $\F_{k,(P_1,\ldots,P_n),\e}$ passing through these points is zero-dimensional by \eqref{eq:dimFkn}, and the curves correspond to $g^1_k$s
satisfying $(\star)$. Hence, as above, this finite family produces, by the bijection \eqref{eq:alphas}, the same curve $C$ in $V(\alpha_1,\ldots,\alpha_p)$, which lies in $V^k_{n,\e}(\alpha_1,\ldots,\alpha_p)$, having finitely many $g^1_k$s with the marked points of ramification being $P_1,\ldots,P_n \in \Gamma \subset C$. This means that $\sigma_{k,n,\e}$ is generically finite. 
Varying $P_1,\ldots, P_n$, the total family of pointed curves produced is now
\[ \dim \F_{k,\e}\stackrel{\eqref{eq:dimFk}}{=} \widetilde{\rho}+g+n,\]
and equals $V^k_{n,\e}(\alpha_1,\ldots,\alpha_p)$. This proves the first assertion in (ii).

{\bf The subcase $n+\widetilde{\rho} \geq 0$.} The same argument starting with $\F_{k,\e}$ shows that for a general set of $\alpha_j$ points on each $\mathfrak{c}_j$, the family of curves in $\F_{k,\e}$ passing through these points has dimension 
\[
\dim \F_{k,\e}- \sum \alpha_j\stackrel{\eqref{eq:dimFk}}{=} \widetilde{\rho}+g+n-\sum \alpha_j \stackrel{\eqref{eq:genere}}{=} \widetilde{\rho}+n \geq 0.
\]
Hence, as above, we get a family of curves in $V^k_{n,\e}(\alpha_1,\ldots,\alpha_p)$, having a
$(\widetilde{\rho}+n)$-dimensional family of  
$g^1_k$s with the given ramification order at {\it some} (varying) points. The total family of (unpointed) curves produced is therefore
\[\dim \F_{k,\e}- \left(\widetilde{\rho}+n\right)\stackrel{\eqref{eq:dimFk}}{=} \widetilde{\rho}+g+n-\left(\widetilde{\rho}+n\right)=g,\]
and lies in $V(\alpha_1,\ldots,\alpha_p)$, so it has to cover the whole of it. This finishes the proof of case (ii-a).

{\bf The subcase $n+\widetilde{\rho} < 0$.} The same argument starting with $\F_{k,\e}$ shows that we get a family of curves in $V^k_{n,\e}(\alpha_1,\ldots,\alpha_p)$, having finitely many
$g^1_k$s with the given ramification order at {\it some} (varying) points. Thus, the forgetful map $V^k_{n,\e}(\alpha_1,\ldots,\alpha_p) \to V(\alpha_1,\ldots,\alpha_p)$
is generically finite. Its image lies in $V^k(\alpha_1,\ldots,\alpha_p)$.
\end{proof}

\begin{proof}[Proof of Theorem \ref{thm:main}]
  As mentioned above, we may assume $p \geq 3$. First of all we note that since $V^k_{|H|,\delta}=\emptyset$ if \eqref{eq:boundintro} is not satisfied, we also have that $V^k_{|H|,\delta,\e}=\emptyset$ if \eqref{eq:boundintro} is not satisfied.

  Now assume that the
 condition \eqref{eq:boundintro} holds. 
 Proposition \ref{prop:bordo} implies that \linebreak  $V^k_{n,\e}(\alpha_1,\ldots,\alpha_p) \subset V^k_{\delta,\e} \cap \pi_{\delta,n}^{-1}(0)$ has a component of the expected dimension, for some $\alpha_i$ fulfilling \eqref{eq:condex2} and \eqref{eq:condex3}. Such $\alpha_i$s exist whenever \eqref{eq:boundintro} is satisfied, as shown in \cite{CK}. The theorem now follows by semicontinuity and Proposition \ref{prop:bordo}. To be more precise, one uses Lemma \ref{lemma:NN} to show non-neutrality of the nodes.
Moreover, to prove the assertion about simple ramification, note that the limit ramification points of a general pencil are 
 the ones of  $\mathfrak g$ on the sectional component $\Gamma$ plus the ones tending to the nodes of $\overline C$.  The former ramification is simple by Proposition \ref{prop:bordo} and the latter is simple because, in the admissible cover setting,  each node is replaced
by a $\PP^ 1$ joining the two branches and mapping $2:1$ to a $\PP^ 1$, hence the ramification is simple there. 
\end{proof}

\section{Proofs of Theorem \ref{thm:main_cycle} and of Corollaries \ref{cor:BV} and \ref{cor:corBV}}

  \label{sec:Zk}

Let $(S,H)$ be a primitively polarized $K3$ surface of genus $p$ and $0 \leq \delta<p$.  Let $CH_0(S)$ be the Chow group of $0$-cycles on $S$. 
  We note that
  \begin{eqnarray}
  \label{eq:I} & \mbox{for any $(p,q)\in Z_{k,\delta}(S,H)$, the points $p$ and $q$ are rationally equivalent on $S$;} & \\
   \label{eq:II} &\mbox{the dimension of every component of $Z_{k,\delta}(S,H)$ is at most two.} & 
    \end{eqnarray}
    Indeed, for any $(p,q) \in Z_{k,\delta}^{\circ}(S,H)$ the image of $k[p-q]=0\in \Jac(C^{\nu})$ via the map $\Jac(C^{\nu}) \to CH_0(S)$ must be 0. As $CH_0(S)$ is torsion free, we conclude that \linebreak $[p-q]=0\in CH_0(S)$. By specialization, it also holds for all $(p,q) \in Z_{k,\delta}(S,H)$, proving \eqref{eq:I}. Consequently, $Z_{k,\delta}(S,H)$ is a subvariety of the fiber over $0$ of the difference map $S \times S\to CH_0(S)$. Now \eqref{eq:II} follows from Mumford's proof of the bound on the fiber dimension \cite{Mum}, which says that the fiber is a countable union of varieties of dimension at most two.

    \begin{lemma}\label{lem:dimgeq2} All components of $Z'_{k,\delta}(S,H)=V^k_{|H|,\delta,k,k}$ have dimension at least $2$. 
    \end{lemma} 
\begin{proof} Set $\mathbf{e}=(k,k)$. The forgetful morphism $\kappa_{g,k, \mathbf{e}}$ from \eqref{eq:kappa} is bijective since the fiber over any $(C,p,q)$ is the unique $g^1_k$ generated by $kp$ and $kq$. It follows by  \eqref{eq:dimG1k} that 
  $\M_{g,k,\e}$ is equidimensional of codimension $-\widetilde{\rho}(g,1,k;k,k)=g$ in $\M_{g,2}$. Via the moduli map \eqref{eq:modgn} we conclude that all components of $V^k_{|H|,\delta, \mathbf{e}}$ have codimension at most $g$ in $V^k_{|H|,\delta,2}$. 
\end{proof}

\begin{lemma}\label{lem:finite} The forgetful map $Z'^{\circ}_{k,\delta}(S,H) \to Z^{\circ}_{k,\delta}(S,H)$ is finite.
\end{lemma}
\begin{proof} Assume the map is not finite. Then there would exist $p,q\in S$ and a one dimensional family of curves $\{(C_t,p,q)\}\in V^k_{|H|,\delta, \mathbf{e}}$ with total ramification at the same points $p,q$. Via the incidence variety in $Z\times \Sym^k(S)$, this yields a one dimensional family of rational curves in $\Sym^k(S)$ all containing the points $kp$ and $kq$. By bend-and-break the closure of this family contains a reducible member. Pulling back via the incidence variety, this produces a reducible curve in the limit of the family $\{C_t\}$ on $S$ (see e.g. the proofs of \cite[Lem. 2.1 and Prop. 4.3]{FKP}). This is not possible for a general $(S,H)\in \K_p$.
\end{proof}

\begin{proof}[Proof of Theorem \ref{thm:main_cycle}]
We note that the condition $\widetilde{\rho} \geq -g $  is automatically satisfied when $n=2$ and $e_1=e_2=k$, whence Theorem \ref{thm:main}  shows that $V^k_{|H|,\delta,k,k}=Z'_{k,\delta}(S,H)$ is nonempty  if and only if \eqref{eq:boundintro}  is satisfied. Lemmas \ref{lem:dimgeq2} and \ref{lem:finite} together with \eqref{eq:II} show that both $Z'_{k,\delta}(S,H)$ and $Z_{k,\delta}(S,H)$ have dimension $2$ when nonempty. 
\end{proof}

\begin{proof}[Proof of Corollary \ref{cor:BV}]
  Since $Z_{k,\delta}(S,H)$ is equidimentional of dimension $2$ in $S\times S$, on each of its components either one of the projections is dominant or both projections map to curves. 
  Hence, the assumptions of \cite[Lem. 2.1]{Tor23} are satisfied and therefore $Z_{k,\delta}(S,H)_*$ admits a lifting to $Z_1(S)$. One can proceed analogously to check that the same holds for $Z_0(S)$. 
The properties (i) and (ii) in Definition \ref{def:pres} are now direct consequences of
 \eqref{eq:I}.

We finally prove the last assertion. If one of the projections to $S$ is dominant, we conclude by \eqref{eq:I}. Otherwise, both projections map to curves, say to $C_1$ and $C_2$, respectively. For any $x \in C_1$, the fibre  $p_1^{-1}(x)$ is a curve, mapping birationally to $C_2$ via $p_2$. By \eqref{eq:I}, any point $y \in C_2$ is rationally equivalent to $x$. As $C_2$ intersects any rational curve in $|H|$, all points $y$ have class $c_S$, whence also $x$ has class $c_S$. 
\end{proof}

    \begin{proof}[Proof of Corollary \ref{cor:corBV}]
        As  every component of $Z_{k,0}(S,H)$  contains a dense set of points $(p,q)$ defining the class $(c_S,c_S)$ by Corollary \ref{cor:BV}, the same holds for the marked points of $(C,p,q)$ in every component of ${Z'}_{k,0}(S,H)$. As $2\leq g$, by \cite[Thm. 1.5]{PS19} this constructs a dense set of tautological points $(C,p,q)$ in the image of $m_{g,2}$ as wanted.
    \end{proof}

Corollary \ref{cor:corBV} constructs subvarieties of dimension $\leq 2$ of $\mathcal{M}_{g,2}$ containing a dense set of tautological points. One can find higher dimensional cycles by considering the relative cycles obtained by moving the polarized K3 surface $(S,H)$ in $\mathcal{K}_p$ and considering the universal pointed Severi variety ${\mathcal{V}_{p,\delta,n}} \to \K_p$. Let $M_{p,\delta,n}:\mathcal{V}_{p,\delta,n}\to \mathcal{M}_{g,n}$ be the map fiberwise defined as $m_{g,n}$ (cf. \eqref{eq:modgn}). Let $\mathcal{Z}'_{k,\delta}\subset{\mathcal{V}_{p,\delta,n}}$ be the cycle defined fiberwise as  ${Z'}_{k,\delta}(S,H)$.  From Corollary \ref{cor:corBV}  we conclude  that  
the set of tautological points in $M_{p,0,2}(\mathcal{Z}'_{k,0})$ is dense. 
   Moreover, we remark that this   cycle has dimension exactly $21$ for  $p=g=11$  or $g \geq 13$. To explain this latter fact, denote by $\mathcal{V}_{p,\delta} \to \K_p$ the universal (unpointed) Severi variety, with fiber over $(S,H)$ being $V_{|H|,\delta}$, and by ${M}_{p,\delta}:\mathcal{V}_{p,\delta}\to \mathcal{M}_{g}$  the map sending an element $(S,H,C) \in \mathcal{V}_{p,\delta}$ to the class of $C^{\nu}$. Then we have a commutative diagram
  \[ \xymatrix{ \mathcal{V}_{p,\delta,n} \ar[r]^{M_{p,\delta,n}} \ar[d] & \mathcal{M}_{g,n} \ar[d] \\
      \mathcal{V}_{p,\delta} \ar[r]^{M_{p,\delta}} & \mathcal{M}_{g,n},}
    \]
with the vertical maps being the finite maps forgetting the points. 
When $\delta=0$, the kernel of the differential of $M_{p,0}$ at a point $(S,H,C)$ is $H^1(\T_S(-C))=H^1(\T_S(-H))$ (cf. \cite{Beau}), which is independent of the curve $C \in V_{|H|,0}$  for fixed $(S,H) \in \K_p$. As ${M}_{p,0}$ is known to be generically finite for $p=11$ and $p \geq 13$, it follows that $H^1(\T_S(-H))=0$ for general $(S,H)$. Thus, the differential of $M_{p,\delta}$ is injective at any point of $(S,C) \in V_{|H|,0}$ for general $(S,H)$, whence $M_{p,\delta}$ is finite at these points. In particular, the map $M_{p,0,2}$ restricted to $\mathcal{Z}'_{k,0}$ is generically finite.

  \end{document}